\documentclass[reqno,12pt]{amsart}

\usepackage{latexsym}
\usepackage{amsmath,amsfonts,amssymb,epsfig}
\usepackage{amsthm,amscd}
\usepackage{mathrsfs} 
\usepackage{hyperref}
\usepackage{cancel}
\usepackage{wrapfig}
\usepackage{array}
\usepackage{float}
\usepackage{sidecap}
\usepackage{color}

\newtheorem{theorem}{Theorem}

\newtheorem{remark}[theorem]{Remark}
\newtheorem{corollary}[theorem]{Corollary}
\newtheorem{proposition}[theorem]{Proposition}
\newtheorem{lemma}[theorem]{Lemma}

\newtheorem{problem}{Problem}

\newenvironment{dedication}
    {\vspace{6ex}\begin{quotation}\begin{center}\begin{em}}
    {\par\end{em}\end{center}\end{quotation}}

\newcommand{\no}{\noindent}

\def\R{\mathbb{R}}
 
\def\Z{\mathbb{Z}}

\def\Pm{\mathbb{P}}

\def\tQ{\mathtt{Q}}

\def\tChi{\pmb{\chi}}

\def\tr{\mathtt{r}}
\def\ts{\mathtt{s}}
\def\tk{\mathtt{k}}
\def\tl{\mathtt{l}}
\def\tX{\mathtt{X}}
\def\NA{\mathcal{N}(\mathsf{A})} 

\graphicspath{{./pict/}}

\numberwithin{equation}{section} 
\numberwithin{theorem}{section}

\title[Random coverings and the Euler characteristic]{Finite random coverings
of one--complexes\\ and the Euler characteristic
}

\author[R. Komendarczyk]{R. Komendarczyk}

\address{Department of Mathematics,
Tulane University,
New Orleans, LA 70118} \email{rako@tulane.edu}

\author[J. Pullen]{J. Pullen}

\address{Department of Mathematics,
Mercer University,
Macon, GA 31207} \email{pullen\_j@mercer.edu}

\date{\today}

\begin{document}
\setcounter{section}{0}

\begin{abstract}
 This article presents an algebraic topology perspective on the problem of finding a complete coverage probability of a one dimensional domain $X$ by a random covering, and develops techniques applicable to the problem beyond the one dimensional case. In particular we obtain a general formula for the chance that a collection of finitely many compact connected random sets placed on $X$ has a union equal to $X$. The result is derived under certain topological assumptions on the shape of the covering sets (the covering ought to be {\em good}, which holds if the diameter of the covering elements does not exceed a certain size), but no a priori requirements on their distribution. An upper bound for the coverage probability is also obtained  as a consequence of the concentration inequality.  The techniques rely on a
formulation of the coverage criteria in terms of the Euler characteristic of the nerve complex associated to the random covering.
\end{abstract}

\maketitle

\vspace{-1.5cm}
\begin{dedication}
{Dedicated to Professor Yuli Rudyak, on the occasion of his 65th birthday.}
\end{dedication}

\section{Introduction}\label{sec:introduction}

We consider {\em finite random coverings} of a metric space $X$, i.e. finite collections of {\em compact random sets}: $\mathsf{A}=\{\mathsf{A}_{\{i\}}\}$, $i=1,\ldots, n$,  understood as  measurable maps \cite[p. 121]{Hall88}
\[
 \mathsf{A}_{\{i\}}:(\Omega,\sigma,\Pm)\xrightarrow{\qquad\qquad} (\mathcal{C}(X)\sqcup \{\text{\rm \O}\},\sigma_{\text{Borel}}),
\]
where $(\Omega,\sigma,\Pm)$ is an underlying probability space, and $\mathcal{C}(X)$ is the set of nonempty compact subsets of $X$, topologized by the Hausdorff distance and given the associated Borel algebra $\sigma_{\text{Borel}}$. The set $\mathcal{C}(X)\sqcup \{\text{\rm \O}\}$ is a disjoint union with the point $\{\text{\rm \O}\}$, which plays a role of the empty set.  The term {\em covering} may be misleading in this context, as it sometimes assumes that the union of its elements contains the domain $X$. In this work a covering is simply a collection of subsets of $X$, as it has been previously used e.g. in \cite{Flatto73, Flatto-Newman77}.

A typical example of an infinite random covering is a {\em coverage process} on the Euclidean space, \cite{Hall88} i.e. a sequence of random sets: $\mathsf{A}=\{\xi_1+G_1,\xi_2+G_2,\cdots,\xi_k+G_k,\ldots \}$, where $\{G_k\}$ is a fixed family of subsets of $\R^n$ called {\em grains} of the process, and $\pmb\xi=\{\xi_i\}$ a sequence of random vectors in $\R^n$. In applications  $G_i$'s are often round balls of a fixed radius, and $\pmb\xi$ defines a Poisson process (in which case $\mathsf{A}$ is refereed to as a {\em Boolean model}, \cite{Hall88}). In the current paper, we make no a priori assumptions on the distribution of $\mathsf{A}$ except a topological requirement on the covering, namely it almost surely must be {\em good}, which means that
each intersection $\bigcap_{i\in I} \mathsf{A}_{\{i\}}$, $I=\{i_1,\ldots, i_k\}$ is almost surely contractible (in general some form of convexity of $\mathsf{A}_{\{i\}}$ validates this assumption).

\begin{problem}\label{pr:main-problem}
Given a random covering $\{\mathsf{A}_{\{i\}}\}$, $i=1,\ldots,n$ of a metric space $X$, find a complete coverage probability: $\Pm\bigl(X\subseteq |\mathsf{A}|\bigr)$, where $|\mathsf{A}|=\bigcup_i\mathsf{A}_{\{i\}}$.
\end{problem}

Reviewing the history of Problem \ref{pr:main-problem}: it was first considered by Whitworth, \cite{Whitworth1897}  in the basic case of a finite collection of independent identically distributed fixed $\alpha$--length  arcs on a unit circumference circle.  Much later, Stevens \cite{Stevens39}  provided a complete answer to the question of Whitworth in the form
 \begin{equation}\label{eq:coverage-S^1}
  \mathbb{P}(S^1\subseteq |\mathsf{A}|)=\sum^{\lfloor \frac{1}{\alpha} \rfloor}_{j=1} (-1)^{j+1} {n \choose j}\Bigl(1-j \alpha\Bigr)^{n-1}\ . 
 \end{equation}
The Stevens' result was further improved by Siegel and Holst \cite{Siegel-Holst82} where they
allowed varying lengths for the arcs. In \cite{Flatto73},  Flatto obtained an asymptotic expression for coverage as $\alpha\to 0$. 
The extension of the circle problem to the 2-sphere $S^2$ was considered by Moran and Groth \cite{Morgan-Groth62}, who derived an approximation for the probability $\mathbb{P}(S^2\subset |\mathsf{A}|)$, and later Gilbert \cite{Gilbert65} showed the bounds 
\[
(1-\lambda)^n \leq P(S^2\subseteq |\mathsf{A}|)\leq \frac 43 n(n-1)\lambda(1-\lambda)^{n-1},
\]  
where $\lambda=(\sin\frac{\alpha}{2})^2$ is the fraction of the surface of $S^2$ covered by spherical $\alpha$--caps; i.e. caps of radius 
$\alpha$. For $\alpha\in [\frac \pi2,\pi]$, the explicit expression for $P(S^2\subseteq |\mathsf{A}|)$ has been found by Miles \cite{Miles69}. Work in \cite{Burgisser-Cucker-Lotz2010} provides explicit formulas for the complete coverage probability for $\alpha$--caps on the $m$-dimensional unit sphere $S^m$ when $\alpha\in [\frac \pi2,\pi]$ and upper bounds for $\alpha\in [0,\pi)$.  The literature concerning the  coverage probability in the asymptotic regimes (where the diameter of grains tends to zero) is vast and we only list a small fraction here \cite{Flatto-Newman77, Siegel79, Athreya-Roy-Sarkar04, Molchanov-Scherbakov03, Bollobas-Bela-Riordan11}. Further, the reader may consult the recent work in \cite{Burgisser-Cucker-Lotz2010} for a more accurate account of the history of Problem \ref{pr:main-problem}.

In this work we focus solely on the case of finite coverings of $1$-dimensional domains $X$ which are homeomorphic to 
finite multigraphs, equipped with an intrinsic distance
\begin{equation}\label{eq:d_X}
 d_X(x,y)=\min_{\substack{\gamma:[0,1]\mapsto X,\\ \gamma(0)=x,\ \gamma(1)=y.}} \text{\rm length}(\gamma).
\end{equation}
I.e. $d_X(x,y)$ is the length of the shortest path between $x$ an $y$, which in practice is just a smallest sum of edge-lengths (and their pieces) connecting $x$ and $y$ (the lengths come from some choice of geometric realization of $X$ in $\R^3$).
Let $\partial X$ denote the set of {\em leaf vertices} of $X$, 
(this notation is justified by the case when $X$ is an interval in $\R$) and $\text{diam}(Y)$ the intrinsic diameter of a subset $Y\subseteq X$.
Our random covering  $\mathsf{A}=\{\mathsf{A}_{\{i\}}\}$ on $X$, will always be finite ($i=1,\ldots, n$). 
As already mentioned before, the basic example of a random covering is $\epsilon$-balls: $\{B(\pmb\xi_i,\epsilon)\}$, $i=1,\ldots, n$ in the intrinsic metric $d_X$, with centers $\pmb\xi_i$ distributed in an arbitrary fashion.  We approach the coverage problem 
by considering a random complex $\mathcal{N}(\mathsf{A})$ directly obtained from  the usual topological nerve (c.f. \cite{Wallace70}) of realizations of $\mathsf{A}$ and its Euler characteristic $\tChi(\mathsf{A})=\tChi(\mathcal{N}(\mathsf{A}))$. 

Let $\mathfrak{C}_n$ be the set of labeled abstract subcomplexes on $n$ vertices (i.e. subcomplexes of the full $(n-1)$--simplex $\pmb\Delta_n$). By the {\em labeling} we understand that every subcomplex $\ts\in \mathfrak{C}_n$ comes with an indexing of its vertices by numbers from $1,\ldots, n$.  Elements  $\ts,\tr,\tk\in \mathfrak{C}_n$ can be identified, in a non-unique way, with subsets of the power set $2^{[n]}$, $[n]=\{1,\ldots,n\}$ (see Section \ref{sec:rand-compx}.)  For instance, a singleton $\tr=\{I\}$ (where  $I\subseteq\{1,\ldots, n\}$) labels a face of  $\pmb\Delta_n$.  
By a {\em finite random complex} on $n$ vertices we understand an arbitrary discrete probability space $\mathsf{K}=\bigl(\mathfrak{C}_n,\mathbb{P}_\mathsf{K}\bigr)$. In order to define  the random nerve $\mathcal{N}(\mathsf{A})$, one builds a distribution on $\mathfrak{C}_n$ in a way dictated by the usual nerve construction. For instance, the probability of a $k$--face $I=\{i_1,\ldots,i_{k+1}\}$ in $\mathfrak{C}_n$, we denote by $p_I$ equals
\begin{equation}\label{eq:p_I}
 p_{I}=\Pm(\{\ts\in \mathfrak{C}_n\ |\ I\in \ts\})=\Pm\bigl(\mathsf{A}_{\{i_1\}}\cap\mathsf{A}_{\{i_2\}}\cap\ldots \cap\mathsf{A}_{\{i_{k+1}\}}\neq \text{\rm \O}\bigr).   
\end{equation}
 Identity \eqref{eq:p_I} can then be extended from faces to subcomplexes. I.e. given $\ts\in\mathfrak{C}_n$ 
\begin{equation}\label{eq:p_s}
\begin{split}
p_{\ts} & =\Pm(\{\tr\in \mathfrak{C}_n\ |\ \ts\subseteq \tr\})=\Pm\bigr(\forall_{I\in \ts}\bigl\{\bigcap_{i\in I} \mathsf{A}_{\{i\}}\neq \text{\rm \O}\bigr\}\,\bigl),\\
P_{\ts} & =\Pm(\ts)=\Pm\bigr(\forall_{I\in \ts}\bigl\{\bigcap_{i\in I} \mathsf{A}_{\{i\}}\neq \text{\rm \O}\bigr\},\forall_{\{J\}\not\in \ts}\bigl\{\bigcap_{j\in J} \mathsf{A}_{\{j\}}= \text{\rm \O}\bigr\}\,\bigl).
\end{split}
\end{equation}
\no  Generally, we make an underlying assumption that the covering $\mathsf{A}$ is {\em good} i.e. to satisfy (almost surely) the hypotheses of the Nerve Lemma (c.f. Section \ref{sec:topological}). In Proposition \ref{prop:good-rips-eps-cover} of Section \ref{sec:nerve-A-rips}, we show that a covering $\mathsf{A}$ of a $1$-complex $X$ is always good, if its elements $\mathsf{A}_{\{i\}}$ are  connected and sufficiently small in diameter. The first version of our main theorem is stated below.

\begin{theorem}[Coverage probability for compact connected 1--complexes  $X$, with $\partial X=\text{\rm \O}$]\label{thm:euler-coverage}
 Let $\mathsf{A}=\{\mathsf{A}_{\{i\}}\}$, $i=1,\ldots,n$ be a random good covering of $X$ (with $\partial X=\text{\rm \O}$). Then, the range of $\tChi=\tChi(\mathsf{A})$ can be restricted to 
\begin{equation}\label{eq:range-chi-rel}
 \underline{m}=\chi(X)\leq \tChi(\mathsf{A})\leq n=\overline{m},
\end{equation}
and the complete coverage probability equals
\begin{equation}\label{eq:cover-prob-euler}
\begin{split}
 \Pm(X\subseteq |\mathsf{A}|) & =\Pm\bigl(\tChi(\mathsf{A})=\chi(X)\bigr)\\
 & =\sum_{\ts\in \mathfrak{C}_n; \chi(\ts)=\chi(X)} P_{\ts}=\sum_{\ts\in \mathfrak{C}_n} a_{\ts}(\tChi)\,p_{\ts},
 \end{split}
\end{equation}
where
\[
a_\ts(\tChi) =\sum_{k=0}^{N}v_{k}(\tChi)\,c_{\ts,k}(\tChi),\quad N=\overline{m}-\underline{m},
\]
with $p_{\ts}$ given in \eqref{eq:p_s}, and 
\begin{align}\label{eq:v-coeff-euler}
v_{k}(\tChi)  & = 
\frac{(-1)^{k}}{N!}\,\sum_{\underline{m}<j_1<j_2<\ldots<j_{N-k}\leq\overline{m}} j_1 j_2 \ldots j_{N-k},\quad\quad v_{N}(\tChi)=\frac{(-1)^{N}}{N!}, \\
\notag c_{\ts,k}(\tChi) 
& =\begin{cases}
{\displaystyle \sum^{r^+_{top}(\ts)}_{i=0}\sum^{r^-_{top}(\ts)}_{j=0} (-1)^{r_{top}(\ts)-i-j}{r^+_{top}(\ts) \choose i}{r^-_{top}(\ts) \choose j}\Bigl(i-j+r^+_{low}(\ts)-r^-_{low}(\ts)\Bigr)^{k}},  & \\
 \qquad \text{if}\ k\geq r(\ts), & \\
0, \\ 
\qquad \text{if}\ k<r(\ts),  
\end{cases}
\end{align}
where  $r^{\pm}=r^{\pm}(\ts)$, $r^{\pm}_{top}=r^{\pm}_{top}(\ts)$, $r^{\pm}_{low}(\ts)=r^\pm-r^{\pm}_{top}(\ts)$ stand for a number of respectively total, top and lower: even(odd) dimensional faces of $\ts\in \mathfrak{C}_n$, and $r(\ts)$ denotes a number of all faces. 

 If the diameter of $\mathsf{A}_{\{i\}}$ is smaller than
$\frac{1}{6} \mathcal{C}$ almost surely (where $\mathcal{C}$ is a length of the shortest cycle in $X$, known as \emph{girth}), then $p_\ts$ further simplifies as
\begin{equation}\label{eq:p_s-rips}
 p_{\ts}=\Pm\bigr(\forall_{(i,j)\in E(\ts),i<j}\bigl\{\mathsf{A}_{\{i\}}\cap \mathsf{A}_{\{j\}}\neq \text{\rm \O}\bigr\}\,\bigl),
\end{equation}
where $E(\ts)$ is the edge set of $\ts$.
\end{theorem}
\no An extension of the above result to the case of a 1--complex $X$ with no assumptions on 
$\partial X$ is provided in Theorem \ref{thm:euler-coverage-rel} of Section \ref{sec:complete-cover-prob}. One obvious corollary of the above result is the fact that the complete coverage probability of any good random covering $\{\mathsf{A}_{\{i\}}\}$ is determined by finitely many numbers, which is not obvious when considering e.g. {\em vacancy}, i.e. the volume of $X-|\mathsf{A}|$ c.f. \cite{Hall88}. The complexity of computing $\Pm(X\subseteq |\mathsf{A}|)$ via the formula of Theorem \ref{thm:euler-coverage} is not addressed here. However, one may expect that, due to the size of the set $\mathfrak{C}_n$, computation of coefficients $a_\ts(\tChi)$ or the set $\{\ts\in \mathfrak{C}_n\ |\ \tChi(\ts)=\chi(X)\}$ is double exponentially hard in $n$. On a positive note, coefficients $a_\ts(\tChi)$ are independent of the underlying distribution vector $(p_\ts)$, therefore once computed for a certain size problem can be reapplied as $(p_\ts)$ changes. The vector $(p_\ts)$ can be conveniently estimated numerically (e.g. via the standard maximum likelihood estimation, c.f. \cite{Lehmann-book}) but again in the simplest case of Equation \eqref{eq:p_s-rips} it is of exponential size:  $2^{{n \choose 2}}$. 
Therefore, in practical situations the formula derived in Theorem \ref{thm:euler-coverage} can apply to the covering problems with small $n$. 

In a longer perspective, one may be interested asymptotic distributions of $\tChi(\mathsf{A})$ (as $n\to \infty$) which would lead to parametric estimators or useful bounds for $\Pm(X\subseteq |\mathsf{A}|)$. Currently available results (e.g. in \cite{Kahle-Meckes10,Kahle09a}) concern 
sparse regimes and they are not applicable, unless we allow the diameter of random sets in $\mathsf{A}=\{\mathsf{A}_{\{i\}}\}$ to 
tend to $0$ sufficiently fast as $n$ tends to infinity (see e.g. \cite{Flatto-Newman77}). Concerning the question of useful bounds for $\Pm(X\subseteq |\mathsf{A}|)$, as a first step we derive an upper bound for the coverage probability, via the concentration inequality \cite{Azuma67} in the following 
\begin{theorem}\label{thm:coverage-upperbounds}
 Let $\mathsf{A}=\{\mathsf{A}_{\{i\}}\}$, $i=1,\ldots,n$ be a random good covering of $X$,  then
\begin{equation}\label{eq:cov-est-rips}
\Pm(X\subseteq |\mathsf{A}|)\leq \exp\Bigl(\frac{-\mu^2_0}{2n(|\chi_{rel}(X,\partial X)|+2)^2}\Bigr),
\end{equation}
where $\mu_0$ denotes the expected value of the relative Euler characteristic $\tChi_{rel}(\mathsf{A},\mathsf{A}_{\partial X})$ of the random pair~$(\mathcal{N}(\mathsf{A}),\mathcal{N}(\mathsf{A}_{\partial X}))$.
\end{theorem}

Although, Theorem \ref{thm:euler-coverage} is restricted to the case of $1$--complexes, the question of complete coverage probability for such spaces is not without a practical meaning. One may consider the $1$--complex to be e.g. a system of streets in the city or underground channels. In such cases random coverings can be associated with sensing regions of e.g. vehicles equipped with sensors (c.f. \cite{deSilva-Ghrist06}). 
 Beyond $1$--complexes, techniques of algebraic topology provide coverage criteria for higher dimensional objects. For instance, if an underlying space $X$ is an $m$-dimensional manifold a necessary and sufficient condition for coverage is nonvanishing of the $m$-th(top) Betti number of the nerve $\mathcal{N}(\mathsf{A})$. We aim to develop these ideas in subsequent papers. 

 The article is organized as follows: In Section \ref{sec:r-complex} we further discuss  the general setup of random complexes and their associated invariants -- mainly $\tChi(\mathsf{A})$. In Section \ref{sec:moments-dist}, we derive relevant formulas for distributions of the random relative Euler characteristic, see Corollary \ref{cor:euler-dist-rel}. Further, in Section \ref{sec:topological}, we prove several basic topological results showing that the Euler characteristic of the nerve of a covering determines complete coverage of a 1--complex. In Proposition \ref{prop:good-rips-eps-cover}, we also provide a sufficient condition for a covering to be good (in terms of the girth of a $1$--complex). We collect relevant facts  and prove Theorem \ref{thm:euler-coverage} in Section \ref{sec:complete-cover-prob}. The upper bound for $\Pm(X\subseteq |\mathsf{A}|)$ of Theorem \ref{thm:coverage-upperbounds} is shown in Section \ref{sec:estimates}.

\subsection*{Acknowledgments} We wish to thank Gustavo Didier for discussions and Tewodros Amdeberhan for many useful suggestions and help with combinatorial aspects of  Section \ref{sec:chi-dist}. We are also grateful to Robert Ghrist for introducing us to the subject both via his work \cite{deSilva-Ghrist06, deSilva-Ghrist07b, deSilva-Ghrist07a} with vin de Silva on coverage problems in sensor networks, and inspiring conversations during the {\em First National Forum of Young Topologists} at Tulane University in 2009. The first author would like to thank the organizers of {\em Applied Topology Conference} in  Bedlewo, 2013. Both authors wish to thank the anonymous referee for corrections to the manuscript.

Theorem \ref{thm:euler-coverage} is a part of the doctoral thesis of \cite{Pul12} of the second author.
Both authors acknowledge the support of NSF DMS \#1043009. The first author was also partially supported by 
DARPA YFA N66001-11-1-4132.

\section{Random complexes and their topological invariants.}\label{sec:r-complex}
\subsection{Random complexes}\label{sec:rand-compx}
We refer the reader to \cite{Hatcher02} for background on algebraic topology.
Consider  $\mathbf{\Delta}_n$ to be a full simplex on $n$--vertices indexed by $1,\ldots,n$ (geometrically $\mathbf{\Delta}_n$ is the convex hull of $n$ points given by the standard basis vectors in $\R^{n+1}$). Recall that a $d$-dimensional face in $\mathbf{\Delta}_n$ can indexed by the collection of its $d+1$ vertices: $I=\{i_1, i_2,\ldots, i_{d+1}\}$, where $1\leq i_1<i_2<\ldots <i_{d+1}\leq n$.  Denote the set of all faces of $\mathbf{\Delta}_n$ by $f(n)$, and particularly, $d$--dimensional faces by $f_d(n)$. I.e.
\begin{equation}\label{eq:f(n)-f_d(n)}
f(n) = \{I\ |\ I\subseteq 2^{[n]}\},\qquad f_d(n) = \{I\ |\ I\subseteq 2^{[n]}, |I|=d+1\}.
\end{equation}
Consider the set of all {\it labeled sub-complexes} $\mathfrak{C}_n$ of $\mathbf{\Delta}_n$ union a special point $\{\text{\rm \O}\}$ playing a role of the empty set. By a  {\em labeled sub-complex} we understand a subcomplex of $\mathbf{\Delta}_n$, determined by all its faces with labeling given by vertices of $\pmb\Delta_n$.
A natural set to consider for enumerating labeled subcomplexes is the power set $2^{f(n)}$ of $f(n)$, which is further denoted by $\mathfrak{P}_n$ (we assume $\mathfrak{P}_n$ contains the empty set). 
Here and thereafter, we use notation $\ts$, $\tr$, $\tk$ for elements of both $\mathfrak{C}_n$ and $\mathfrak{P}_n$.

Clearly, there is a surjective correspondence $\Pi:\mathfrak{P}_n\mapsto \mathfrak{C}_n$ which to a given subset $\ts\in \mathfrak{P}_n$ assigns a subcomplex $\Pi(\ts)$ in $\mathfrak{C}_n$ given by the union of elements $I$ of $\ts$ and their subsets (i.e. the lattice of subsets associated to the faces of subcomplex $\ts$). $\Pi$ is clearly not bijective, however, with certain choices we may easily build right inverses. In particular, we will be interested in two cases, which we refer to as the {\it antichain} and {\it chain} representations: $\widehat{\cdot}:\mathfrak{C}_n\mapsto \mathfrak{P}_n$, $\widetilde{\cdot}:\mathfrak{C}_n\mapsto \mathfrak{P}_n$. The antichain representative $\widehat{\ts}\in \mathfrak{P}_n$ of $\ts\in\mathfrak{C}_n$,  contains only its top dimensional faces, also known as {\em facets}, i.e. 
\[
\widehat{\ts}=\{I\in \ts\ |\ \text{such that}\ \text{for any}\ J\in \ts, J\neq I\ \text{either}\ J\subset I\ \text{or}\ (J\not\subseteq I\ \text{and}\ I\not\subseteq J)\}.
\]
The chain representative $\widetilde{\ts}$ of $\ts\in\mathfrak{C}_n$ is obtained from 
the antichain representative  by adding all remaining subfaces of $\ts$. Clearly, $\widetilde{\ts}=\ts$ if $\ts \in  \mathfrak{C}_n$ thus $\mathfrak{C}_n=\widetilde{\mathfrak{C}}_n$, we also have projections $\widehat{\Pi}:\mathfrak{P}_n\mapsto \widehat{\mathfrak{C}}_n$, $\widetilde{\Pi}:\mathfrak{P}_n\mapsto \widetilde{\mathfrak{C}}_n=\mathfrak{C}_n$, where $\widetilde{\Pi}=\Pi$.
Note that the cardinality of $\widehat{\mathfrak{C}}_n$, and therefore $\mathfrak{C}_n$ and $\widetilde{\mathfrak{C}}_n$, is given by the Dedekind number $M(n)$, c.f. \cite{Kisielewicz88}. For any $\ts\in \mathfrak{C}_n$, we call the elements of $\widehat{\ts}$, {\em top faces} or {\em facets} of $\ts$.

Recall from Section \ref{sec:introduction} that by {\em finite random complex} on $n$ vertices we understand a discrete probability space $\mathsf{K}=\bigl(\mathfrak{C}_n,\Pm_\mathsf{K}\bigr)$. 
It is easy to see that $\mathbb{P}_\mathsf{K}$ satisfies the following equivalent conditions\footnote{because events $\{\ts\in \mathfrak{C}_n\,|\,I\in \ts\}$ and $\{\tr\in \mathfrak{C}_n\,|\,\{I'\}\notin \tr\}$ are disjoint} (for $I'\subseteq I$)

\begin{itemize}
\item[{\bf (A)}] $\Pm_\mathsf{K}\bigl(\, I\ |\ (I')^c\,\bigr)=\Pm_\mathsf{K}\bigl(\{\ts\in \mathfrak{C}_n\,|\,I\in \ts\}\ |\ \{\tr\in \mathfrak{C}_n\,|\,\{I'\}\notin \tr\}\,\bigr)=0$,
\item[{\bf (B)}] $\Pm_\mathsf{K}\bigl(\, I'\ |\ I\,\bigr)=\Pm\bigl(\{\ts\in \mathfrak{C}_n\,|\,\{I'\}\in \ts\}\ |\ \{\tr\in \mathfrak{C}_n\,|\,I\in \tr\}\bigr)=1$.
\end{itemize}
In short {\bf (A)} says that if a subface $I'$ of $I$ has not occurred then $I$ cannot occur either; equivalently, {\bf (B)} says that $I'$ occurs whenever $I$ has occurred. 
We say that $\mathsf{K}$ is supported on a subcomplex, $\tk\in \mathfrak{C}_n$ if and only if for any $I\not\in \tk$ we have 
$\mathbb{P}_{\mathsf{K}}(\{I\})=0$.
 Given random complexes $\mathsf{K}$ and $\mathsf{L}$ on $n$-vertices, the joint probability space  $(\mathsf{K},\mathsf{L}):=(\mathfrak{C}_n\times \mathfrak{C}_n,\mathbb{P}_{\mathsf{K},\mathsf{L}})$ is a {\em random pair} if and only if  $\mathsf{L}$ is {\em almost surely a subcomplex} of $\mathsf{K}$, i.e. 
the following condition holds
\begin{itemize}
\item[$\textbf{(C)}$] for every $(\ts,\tr)\in \mathfrak{C}_n\times\mathfrak{C}_n$ such that $\tr\not\subseteq \ts$ we have $\mathbb{P}_{\mathsf{K},\mathsf{L}}(\,\ts, \tr\,)=0.$
\end{itemize}
For a given $\mathsf{K}$ (or $(\mathsf{K},\mathsf{L})$) it will be convenient to consider Bernoulli random variables which are 
{\em indicator functions of faces} in $\mathsf{K}$, i.e. for  $I\in f(n)$ we define
\begin{align}\label{eq:e_I}
 e_I:  \mathfrak{C}_n\longrightarrow \{0,1\},\qquad
 e_I(\ts)=\begin{cases}
 1,& \qquad I\in \ts,\\
 0, & \qquad\text{otherwise},
\end{cases}
\end{align}
For any $\ts\in \mathfrak{C}_n$, we set $e_{\ts}=\prod_{I\in \ts} e_I$ an {\em indicator function of the subcomplex $\ts$}. Clearly $e_{\ts}$ takes value $1$ on $\tr$ if and only if $\ts\subseteq \tr$.
 Let us define vectors $(p_\ts)$ and $(P_\ts)$ (directly related to vectors in \eqref{eq:p_s}, when the underlying random complex is $\mathcal{N}(\mathsf{A})$):
\begin{equation}\label{eq:p_s-P_s-e_I}
\begin{split}
p_\ts & =\Pm\Bigl(e_\ts=\prod_{I\in \ts} e_I =1\Bigr),\\
P_\ts & =\Pm\Bigl(\bigl(\prod_{I\in f(n), I\in \ts} e_I \prod_{J\in f(n), J\not\in \ts}  (1-e_J)\bigr)=1\Bigr).
\end{split} 
\end{equation}

Clearly, a random complex $\mathsf{K}$ is fully determined by indicator functions $\{e_I\}$ of faces and their joint distribution. In the next section $e_I$s will serve as formal indeterminates for functions defining  topological random variables on $\mathsf{K}$, such as the Euler characteristic. The main use of conditions {\bf (A)}({\bf (B)}) and {\bf (C)} is to  define (in Section \ref{sec:rnd-polynomials}) a natural polynomial ring for random topological invariants such as $\tChi(\mathsf{K})$. 

\begin{remark} 
{\rm In general, one could consider a more flexible model of a finite random complex with $(\mathfrak{P}_n,\bar{\Pm})$ as the underlying probability space. It can be thought of as a distribution on open faces of $\pmb\Delta_n$ (i.e. interiors of faces) with an exception of the zero dimension (the vertices.) In this model it is possible, for instance, for an edge to occur without its vertices (i.e. {\bf (A)} can be violated). }
\end{remark}

\begin{remark}
{\em 
 We may easily generalize the definition of the random complex $\mathsf{K}$ to the case $n=\infty$, and thus removing dependence on $n$ in the definition. This is done by considering 
 all labeled subcomplexes $\mathfrak{C}_\infty$ of the infinite simplex $\mathbf{\Delta}_\infty=\bigcup_n \mathbf{\Delta}_n$, and regarding a {\em random complex} $\mathsf{K}$ as a probability space $(\mathfrak{C}_\infty,\mathbb{P}_{\mathsf{K}})$. Such random complex is finite provided the support of $\mathsf{K}$ is contained in $\mathbf{\Delta}_n$ for sufficiently big $n$. 
}
\end{remark}

\subsection{Topological invariants in the random setting} Recall that, thanks to the Poincare-Euler formula \cite{Hatcher02}, the Euler characteristic of a general $n$-complex $K$ is  given by
\begin{equation}\label{eq:chi-abs}
 \chi(K)=\sum^n_{j=0} (-1)^{j}\,\text{\rm dim}\, C_j(K;\R),
\end{equation}
where $\text{\rm dim}\, C_j(K;\R)$ denotes the dimension, as a vector space, of the real coefficient $j$th chain group $C_j(K;\R)$, and equals (in the absolute case) to the number of $j$-dimensional faces $\pmb f_j(K)$ of $K$. We will also need a relative version of $\chi$. Given a pair $(K,L)$ where $L$ is a subcomplex of $K$ we have 
\begin{equation}\label{eq:chi-rel}
 \chi_{rel}(K,L)=\sum^n_{j=0} (-1)^{j}\,\text{\rm dim}\, C_j(K,L;\R),
\end{equation}
where $\text{\rm dim}\, C_j(K,L;\R)$ denotes the dimension of  the $j$th relative chain group $C_j(K,L;\R)= C_j(K;\R)/C_j(L;\R)$, as a real vector space (c.f. \cite{Hatcher02}). Note that 
\begin{equation}\label{eq:rk(X,Y)}
 \text{\rm dim}\, C_j(K,L;\R)=\pmb f_j(K)-\pmb f_j(L).
\end{equation}
Invariants $\chi=\chi(K)$ and $\chi_{rel}=\chi_{rel}(K,L)$ can be expressed in terms of Betti numbers $\{\beta_k(K)\}$, $\{\beta_k(K,L)\}$ of the chain complexes $C_\ast(K)$ and $C_\ast(K,L)$, (c.f. \cite{Hatcher02}). Specifically,
\begin{align}
\label{eq:chiX-betti} \chi & =\sum^n_{j=0} (-1)^{j}\,\beta_j(K),\qquad \chi_{rel} =\sum^n_{j=0} (-1)^{j}\,\beta_j(K,L).
\end{align}
\subsection{Random polynomials}\label{sec:rnd-polynomials} Given a random complex $\mathsf{K}$ let us treat the indicator functions of faces $\{e_I\}$ (or in a case of a random pair $\{e_I,w_J\}$) as formal indeterminates and consider a polynomial ring in $e_I$ (without loss of generality we work over $\R$):
\[
 \R[e_I]:=\R[e_{\{1\}},\ldots,e_{\{n\}},e_{\{1,2\}},\ldots,e_{\{i_1,\ldots,i_k\}},\ldots, e_{\{1,\ldots,n\}}],
\]
or $\R[e_I,w_J]$ in the case of random pairs. Observe that any random variable $\tX$ on $\mathsf{K}$ is given as such polynomial, i.e. 
\begin{equation}\label{eq:P}
 \tX=\sum_{\ts\in\mathfrak{C}_n} \tX_\ts \Bigl(\prod_{I\in f(n), I\in \ts} e_I \prod_{J\in f(n), J\not\in \ts}  (1-e_J)\Bigr),
\end{equation}
where $\tX_\ts$ is a value of $\tX$ at $\ts\in \mathfrak{C}_n$.
Based on \eqref{eq:chi-abs} 
 we may express the random Euler characteristic $\tChi=\tChi(\mathsf{K})$
\begin{equation}\label{eq:rnd-euler-char-def}
\begin{split}
 & \tChi:(\mathfrak{C}_n,\mathbb{P}_{\mathsf{K}})\longrightarrow \Z,\\
 & \tChi(\ts)=\chi(\ts),
\end{split}
\end{equation}
as the following polynomial in $\R[e_I]$: 
\begin{equation}\label{eq:rnd-euler-poly}
\begin{split}
   \tChi & =\sum_{I\in f(n)} (-1)^{|I|-1} e_I.
\end{split}
\end{equation}

\begin{lemma}\label{lem:R_I-ring}
Given a random complex $\mathsf{K}=(\mathfrak{C}_n,\Pm_{\mathsf{K}})$ and its collection of the indicator functions $\{e_I\}$, consider $\tQ, \tQ'\in \R[e_I]$ as two representatives of the same  coset in $\R[e_I]/\mathcal{I}$ where
$\mathcal{I}$ is an ideal generated by the following relations 
\begin{equation}\label{eq:relations-R_I}
 \{e_J\,e_I=e_J\ |\ \text{for all}\ I\subseteq J\},
\end{equation}
(in particular: $e^2_I=e_I$). Then $\tQ=\tQ'$ almost surely.
\end{lemma}
\begin{proof}
 It suffices to show that $\Pm(e_J\,e_I=e_J)=1$ for any $I,J$ where $I\subseteq J$. We have
\begin{align*}
& \Pm(e_J\,e_I=0)=\Pm(e_J=0,e_I=1)+\Pm(e_J=1,e_I=0)+\Pm(e_J=0,e_I=0).
\end{align*}
Thanks to {\bf (A)} : $\Pm(e_J=1,e_I=0)=0$, thus 
\begin{align*}
& \Pm(e_J\,e_I=0)=\Pm(e_J=0,e_I=1)+\Pm(e_J=0,e_I=0)=\Pm(e_J=0),
\end{align*}
and $\Pm(e_J\,e_I=1)=1-\Pm(e_J\,e_I=0)=1-\Pm(e_J=0)=\Pm(e_J=1)$.
\end{proof}
\no We will further denote the quotient ring $\R[e_I]/\mathcal{I}$ by $\R_{\mathcal{I}}[e_I]$. Clearly, $\R_{\mathcal{I}}[e_I]$ has an additive  basis of monomials indexed by the chain representatives: $\ts\in\mathfrak{C}_n$:
\begin{equation}\label{eq:e_J}
 e_{\ts}=\prod_{I\in \ts} e_I.
\end{equation}
\no In the case of pairs $(\mathsf{K},\mathsf{L})$ we have a pair of sets of face indicator functions 
$\{e_I, w_J\}$ corresponding to $\mathsf{K}$ and $\mathsf{L}$
respectively. Then, it is relevant to consider a polynomial ring $\R[e_I,w_J]$ modulo relations 
in \eqref{eq:relations-R_I} and additionally (thanks to property {\bf (C)}):
\begin{equation}\label{eq:relations-R_I-pair}
\begin{split}
 & \{w_J\,w_I=w_J\ |\ \text{for all}\ I\subseteq J\},\\
 & \{ w_I=w_I e_J,\ |\ \text{for all}\ J\subseteq I\}.
\end{split}
\end{equation}
The resulting quotient ring will be denoted by $\R_{\mathcal I}[e_I,w_J]$, and the analogous statement as Lemma \ref{lem:R_I-ring} is true for random variables expressed as representatives  in $\R_{\mathcal I}[e_I,w_J]$. An important for us example of a polynomial in $\R_{\mathcal I}[e_I,w_J]$ is the relative Euler characteristic 
\begin{equation}\label{eq:rnd-euler-char-rel-def}
\begin{split}
 & \tChi_{rel}(\mathsf{K},\mathsf{L}):(\mathfrak{C}_n\times\mathfrak{C}_n,\mathbb{P}_{\mathsf{K}})\longrightarrow \Z,\\
 & \tChi_{rel}(\ts,\ts')=\chi_{rel}(\ts,\ts'),\ \text{if}\ \ts'\subseteq \ts, \quad\text{i.e. the relative Euler characteristic of $(\ts,\ts')$}\\
 & \qquad\quad =0,\ \text{if}\ \ts'\not\subseteq \ts.
\end{split}
\end{equation}
Note, that thanks to {\bf (C)}, the set of pairs $(\ts,\ts')$ such that $\ts'\not\subseteq \ts$ is of measure zero in $(\mathsf{K},\mathsf{L})$ and thus the value 
of $\tChi_{rel}=\tChi_{rel}(\mathsf{K},\mathsf{L})$ on such pairs is irrelevant. Thanks to \eqref{eq:rk(X,Y)}, the polynomial expression for 
$\tChi_{rel}$ is given as  follows 
\begin{equation}\label{eq:rnd-rel-euler-poly}
\begin{split}
   \tChi_{rel} & =\sum_{I\in f(n)} (-1)^{|I|-1} (e_I-w_I).
\end{split}
\end{equation}

\section{Moments and  distributions of the random Euler characteristic.}\label{sec:moments-dist}

We begin with basic review of the {\em method of moments} for the finite range discrete random variable $\tX$, and provide a specific formulation based on the recent work in~\cite{El-Mikkawy03}. Alternatively, one could use factorial moments (see e.g. \cite[p. 17]{Bollobas-book85}), however they do not offer any advantage in the setting of the random Euler characteristic. 

\subsection{Method of moments}\label{sec:moments-method} First, we need basic information on the {\it Vandermonde matrix} $\mathcal{V}$ (c.f. \cite{Meyer00-book}). Given a fixed sequence of real numbers $\mathbf{x}=\{x_0,x_1,\ldots,x_N\}$, $\mathcal{V}$ is an $(N+1)\times (N+1)$ matrix explicitly given as follows
\[
\mathcal{V}=\mathcal{V}(\mathbf{x}) = \left(\begin{array}{cccc} 1 & x_0 & \cdots & x^{N}_0\\
1 & x_1 & \cdots & x^{N}_1\\
\vdots & \vdots & \ddots & \vdots\\
1 & x_N & \cdots & x_N^{N}\end{array}
\right).
\]
\no Note that $\mathcal{V}$ is invertible provided the $x_i$'s are distinct (c.f. \cite{Meyer00-book}). A closed form of $\mathcal{V}^{-1}$ has been derived in \cite{El-Mikkawy03} in terms of the elementary symmetric polynomials. Denote by   
 $\mathfrak{e}_i(j)(\mathbf{x})$ the $i$th--elementary symmetric polynomial in variables: $x_0,\cdots, \widehat{x}_j,\cdots,x_N$ for $j = 0, \cdots, N$, where $\widehat{x}_j$ means that $x_j$ is omitted. Specifically
\begin{eqnarray}\label{eq:sigma-ij}
   \mathfrak{e}_i(j)(\mathbf{x}) = \left\{
     \begin{array}{cl}
       1 & \mbox{if} \  i = 0\\
      \displaystyle{ \sum_{1\leq l_1<l_2<\ldots<l_{i}\leq N; l_k\neq j} x_{l_1}x_{l_2}\ldots x_{l_{i}}} & \mbox{if} \  i > 0\ .
     \end{array}
   \right.
\end{eqnarray}
\no By \cite[p. 647]{El-Mikkawy03}, we have
\begin{equation}
\label{eq:v-x} \mathcal{V}(\mathbf{x})^{-1} =(v_{k i}(\mathbf{x})),\qquad\text{where}\quad
 v_{k i}(\mathbf{x}) = (-1)^{N+k}\frac{\mathfrak{e}_{N-k}(i)(\mathbf{x})}{\prod_{j=0, j\ne i}^N (x_i - x_j)},
\end{equation}
for $i=0,\ldots,N$, $k=0,\ldots,N$. In the case $\mathbf{x}$ is an integer interval $[\underline{m},\ldots,\overline{m}]$, $\underline{m},\overline{m}\in\Z$, $\underline{m}\leq \overline{m}$ of  size $N=\overline{m}-\underline{m}$ we obtain 
\begin{equation}\label{eq:v-0-n}  
v_{k i}(\mathbf{x})=v_{k i}(\underline{m},\overline{m})  = \frac{(-1)^{i+k}}{N!}\,{N \choose i}\,\mathfrak{e}_{N-k}(i)(\underline{m},\ldots,\overline{m}).
\end{equation}
\begin{lemma}\label{lem:X-dist}
Let $\mathtt{X}$ be a discrete random variable of a finite range $\mathbf{x}=\{x_0,x_1,\ldots, x_N\}$, and let $\mu_k=\mathbb{E}(\mathtt{X}^k)$ denote the $k$-th moment of $\mathtt{X}$. Given the vector $\boldsymbol{\mu}=(\mu_0,\ldots, \mu_{N})$ we can recover the distribution of $\mathtt{X}$ explicitly as follows
\begin{align}\label{eq:pk-x}
p_i=\Pm(\mathtt{X}=x_i) &
      =\sum_{k=0}^N v_{k i}\,\mu_{k},\qquad i=0,\ldots,N,
\end{align}
where $v_{k i}=v_{k i}(\mathbf{x})$ are the Vandermonde coefficients. 
\end{lemma}

\begin{proof} By definition we have a linear system of $N$ equations
\[
\mu_k = \sum_{i=0}^N x_i^k p_i,\qquad \mbox{for}\quad k=0,1,\ldots,N.
\]
In matrix form this system reads: $\mathbf{p}\mathcal{V} = {\boldsymbol{\mu}}$ where $\mathbf{p}=(p_0,\ldots,p_N)$, and $\boldsymbol{\mu}=(\mu_0,\ldots,\mu_{N})$. Since all $x_i$'s are distinct $\det(\mathcal{V})=\prod_{i\neq j} (x_i-x_j)\neq 0$. Thus $\mathcal{V}$ is invertible and we have the unique solution $\mathbf{p}=\boldsymbol{\mu}\mathcal{V}^{-1}$.  Identity \eqref{eq:pk-x} is now a direct consequence of \eqref{eq:v-x}.
\end{proof}
\no Our goal for the next subsection is to provide expressions for distributions of polynomial random variables in $\R_{\mathcal{I}}[e_I]$. 
\subsection{Distributions of random polynomials.}\label{sec:poly-dist} Since the differences between  $\R_{\mathcal{I}}[e_I]$ and $\R_{\mathcal{I}}[e_I,w_J]$ are mostly notational, we choose to work with the former.
Recall from Section \ref{sec:rnd-polynomials} that any representative polynomial in $\R[e_I]$ is a linear combination of monomials $e_{\tk}$ from \eqref{eq:e_J}
\begin{equation}\label{eq:Q-expn}
\tQ=\sum_{\tk\in \mathfrak{P}_n} c_{\tk}\, e_{\tk},\qquad c_{\tk}\in\R,
\end{equation}
 where the constant coefficient $c_0=c_{\text{\rm \O}}$ is indexed by the empty set. Note that if $\tQ\in \R_{\mathcal{I}}[e_I]$ then, thanks to the relations in $\R_{\mathcal{I}}[e_I]$, we may always pick expansions of $\tQ$ in terms of the antichain or chain representatives i.e. 
\begin{equation}\label{eq:Q-expn-min-max}
\tQ=\sum_{\widehat{\ts}\in \widehat{\mathfrak{C}}_n} c_{\widehat{\ts}}\,e_{\widehat{\ts}},\qquad \text{or}\qquad  \tQ=\sum_{\widetilde{\ts}\in \widetilde{\mathfrak{C}}_n} c_{\widetilde{\ts}}\,e_{\widetilde{\ts}}=\sum_{\ts\in \mathfrak{C}_n} c_{\ts}\,e_{\ts},
\end{equation}
where in the second expansion we just applied our convention from Section \ref{sec:rand-compx} to identify elements of $\mathfrak{C}_n$ with their chain representatives. We refer to 
\ref{eq:Q-expn-min-max}(left) as the {\em antichain representative} and  \ref{eq:Q-expn-min-max}(right) as the {\em chain representative} of $\tQ$ in $\R_{\mathcal{I}}[e_I]$.
Note that from  Lemma \ref{lem:R_I-ring} it is irrelevant which expansion of $\tQ$ we choose.
 Below, we outline a strategy to determine coefficients $c_{\tk}$ of \eqref{eq:Q-expn} via the inclusion--exclusion principle.

Recall, the general form of the {\em inclusion--exclusion principle}, \cite{Lovasz93}: Given a finite set $F$ and functions $f,g:2^F\longrightarrow \R$, 
\begin{equation}\label{eq:incl-excl-g-to-f}
 g(S')=\sum_{S:S\subseteq S'} f(S),\qquad S'\subseteq F,
\end{equation}
we have 
\begin{equation}\label{eq:incl-excl-f-to-g}
 f(S')=\sum_{S:S\subseteq S'} (-1)^{|S'|-|S|} g(S),\qquad S'\subseteq F. 
\end{equation}

\no Recall the following notation: given $\tQ\in \R[e_I]$ and $\ts\in \mathfrak{P}_n$ define
\begin{equation}\label{eq:Q(s)}
 \tQ(\ts):=\tQ(\{e_{I}=1\ |\ I\in \ts\}).
\end{equation}
I.e. $\tQ(\ts)$ is a polynomial obtained from $\tQ$ by substituting $e_{I}=1$ for all $I\in \ts$, and $\tQ(\ts)(0)$ its constant coefficient.

\begin{lemma}\label{lem:c_J-eval}
 Consider any representative $\tQ\in \R_{\mathcal{I}}[e_I]$ in a general form
\eqref{eq:Q-expn}. For any  $\tk\in \mathfrak{P}_n$ the coefficient 
$c_\tk$ of $\tQ$ in the expansion \eqref{eq:Q-expn} is given as follows
\begin{equation}\label{eq:C_J-eval}
 c_{\tk}(\tQ)=\sum_{\tr\in \mathfrak{P}_n,\tr\subseteq \tk} (-1)^{|\tk|-|\tr|}\, \tQ(\tr)(0).
\end{equation}
 In the case $\tQ$ is represented by the chain expansion (right)\eqref{eq:Q-expn-min-max}, for any $\ts\in \mathfrak{C}_n$, $\ts\neq\{\text{\rm \O}\}$ we have
\begin{equation}\label{eq:c_J-eval-Cn}
 c_{\ts}(\tQ)=\sum_{\tr\in \mathfrak{C}_n,\tr\subseteq \ts} (-1)^{|\ts|-|\tr|}\, (\tQ(\tr)(0)-c_0),
\end{equation}
where $c_0=c_{\text{\rm \O}}=\tQ(0)$ is the constant term of $\tQ$. 
\end{lemma} 
\begin{proof}
In the inclusion--exclusion principle set $F=\tk$. Then any subset $S\subseteq F$ is just a subset of faces $\tr$ of $\tk$, i.e. $\tr\in\mathfrak{P}_n$ and $\tr\subseteq \tk$. Directly from  \eqref{eq:Q-expn} and \eqref{eq:Q(s)} for any $\tr\subseteq \tk$,  we have
\[
 \tQ(\tr)(0)=\sum_{\tr'\subseteq \tr} c_{\tr'}
\]
thus setting $g(\tr)=\tQ(\tr)(0)$ and $f(\tr)=c_{\tr}$, Equation \eqref{eq:C_J-eval} follows from \eqref{eq:incl-excl-f-to-g}.
To obtain \eqref{eq:c_J-eval-Cn} consider the polynomial $\bar{\tQ}=\tQ-c_0$. If $\tr\subseteq \tk$ and $\tr\neq\widehat{\tr}$,  then  $\bar{\tQ}(\tr)(0)=0$. Therefore,  for $\ts\in\mathfrak{C}_n$, Equation \eqref{eq:C_J-eval} yields
\[
 c_{\ts}(\bar{\tQ})=\sum_{\tr\in \mathfrak{P}_n,\tr\subseteq \ts} (-1)^{|\ts|-|\tr|}\, \bar{\tQ}(\tr)(0)=\sum_{\tr\in \mathfrak{C}_n,\tr\subseteq \ts} (-1)^{|\ts|-|\tr|}\, \bar{\tQ}(\tr)(0).
\]
Because $c_{\ts}(\tQ)=c_{\ts}(\bar{\tQ})$ for $\ts\neq \text{\rm \O}$, the identity in \eqref{eq:c_J-eval-Cn} follows.
\end{proof}
\no For a polynomial random variable $\tQ\in \R[e_I]$ in a general form \eqref{eq:Q-expn}, define constants
\begin{align}\label{eq:Q-range-bounds}
&  \underline{m}(\tQ)  =\sum_{\ts\in \mathfrak{P}_n} c^-_{\ts},\quad c^-_{\ts}=\min\{c_{\ts},0\},\qquad
 \overline{m}(\tQ)  =\sum_{\ts\in \mathfrak{P}_n} c^+_{\ts},\quad c^+_{\ts}=\max\{c_{\ts},0\}.
\end{align}
Denote the coefficients of the general expansion \eqref{eq:Q-expn} of the chain representative of the $k$-th power $(\tQ)^k$ by $c_{\ts,k}(\tQ)$, i.e. 
\begin{equation}\label{eq:Q^k}
 \tQ^k=\sum_{\ts\in \mathfrak{C}_n} c_{\ts,k}(\tQ)\, e_\ts\ .
\end{equation}

\no We summarize efforts of this section by stating the following result which is a direct consequence of Lemma \ref{lem:X-dist} and Lemma \ref{lem:c_J-eval}.

\begin{theorem}\label{thm:Q-dist}
Given  $\tQ$ as a chain representative in $\R_{\mathcal{I}}[e_I]$, 
suppose that the set of realizations of $\tQ$ is in the integer interval $[\underline{m}, \overline{m}]$. Then the distribution of $\tQ$ and  its moments are 
 given as follows 
\begin{equation}\label{eq:Q-dist-again}
\begin{split}
\mu_k =\mathbb{E}(\tQ^k)  & = \sum_{\ts\in \mathfrak{C}_n}(\tQ(\ts)(0))^k P_\ts=\sum_{\ts\in \mathfrak{C}_n} c_{\ts,k}(\tQ) p_\ts,\\
\Pm(\tQ=\underline{m}+j) & =\sum_{\ts\in \mathfrak{C}_n; \tQ(s)(0)=\underline{m}+j} P_{\ts}=\sum_{\ts\in \mathfrak{C}_n} a_{\ts,j}(\tQ)\,p_{\ts},\qquad j\in [0,N], \quad N=\overline{m}-\underline{m}\\
 &\qquad \text{for} \qquad a_{\ts,j}(\tQ)=\sum_{k=0}^{N}v_{k j}(\tQ)\,c_{\ts,k}(\tQ),
\end{split}
\end{equation} 
where $v_{k j}(\tQ)$  were defined in \eqref{eq:v-x}. Further, $c_0=\tQ(0)$ and $c_{0,k}=c^k_0$, and for $\ts\neq \text{\rm \O}$:
\begin{equation}\label{eq:c_J,k}
c_{\ts,k}(\tQ)=\sum_{\tr\in \mathfrak{C}_n;\tr\subseteq \ts} (-1)^{|\ts|-|\tr|} (\tQ(\tr)(0)-c_0)^k.
\end{equation}
\end{theorem}
\begin{proof}
 Since $e_{\ts}$ are Bernoulli random variables 
\[
 \mu_k=\mathbb{E}(\tQ^k)=\sum_{\ts\in \mathfrak{C}_n} c_{\ts,k}(\tQ)\, \mathbb{E}(e_{\ts})=\sum_{\ts\in \mathfrak{P}_n} c_{\ts,k}(\tQ)\, p_{\ts},
\]
thus \eqref{eq:Q-dist-again} is an immediate consequence of \eqref{eq:pk-x}.  Formula \eqref{eq:c_J,k} follows from \eqref{eq:c_J-eval-Cn} applied to $\tQ^k$.
\end{proof}
\subsection{Formulas for \texorpdfstring{$\tChi(\mathsf{K})$}{chi}, \texorpdfstring{$\pmb f_d(\mathsf{K})$}{fd} and \texorpdfstring{$\tChi_{rel}(\mathsf{K},\mathsf{L})$}{chi-rel}}\label{sec:chi-dist}
In this section we aim to provide slightly more tractable formulas for the coefficients $c_{\ts,k}(\,\cdot\,)$ and the integer ranges $[\underline{m}(\,\cdot\,), \overline{m}(\,\cdot\,)]$ for the polynomials $\tChi=\tChi(\mathsf{K})$, $\pmb f_d=\pmb f_d(\mathsf{K})$ and $\tChi_{rel}=\tChi_{rel}(\mathsf{K},\mathsf{L})$, where $\mathsf{K}$ is a given random complex on $n$ vertices.   Thanks to Theorem \ref{thm:Q-dist}, it will provide us with a more precise characterization of distributions for these polynomials. 

We begin with the  case of $\pmb f_d(\mathsf{K})$. 
Clearly, the range of $\pmb f_d$ is contained in between 
\begin{equation}\label{eq:range-F_d}
  \underline{m}(\pmb f_d)=0,\qquad\text{and}\qquad
  \overline{m}(\pmb f_d)={n \choose d+1}.
\end{equation}
\no For a subcomplex $\ts\in \mathfrak{C}_n$ and its corresponding antichain $\widehat{\ts}$, recall the following notation 
\begin{equation}\label{eq:faces-notation}
\begin{split}
 r^+_{top}=r^+_{top}(\ts) & =\{\text{numer of even dimensional faces in $\widehat{\ts}$}\},\\
 r^-_{top}=r^-_{top}(\ts) & =\{\text{numer of odd dimensional faces in $\widehat{\ts}$}\},\\
 r^+_{low}=r^+_{low}(\ts) & =\{\text{numer of even dimensional faces in $\ts-\widehat{\ts}$}\},\\
 r^-_{low}=r^-_{low}(\ts) & =\{\text{numer of odd dimensional faces in $\ts-\widehat{\ts}$}\},\\
r_{top}=r_{top}(\ts) & = r^+_{top}+r^-_{top}=|\widehat{\ts}|,\\
r_{low}=r_{low}(\ts) & = |\ts|-|\widehat{\ts}|,\qquad r=r(\ts)  = r_{top}+r_{low}=|\ts|.
\end{split}
\end{equation}

 Given a random complex $\mathsf{K}$, a basic example of interest is the number of its $d$--dimensional faces
\begin{equation}\label{eq:rnd-num-faces-def}
 \pmb f_d=\sum_{\{I\}\in \mathfrak{C}_n; |I|=d+1} e_I,
\end{equation}
\no and the {\em Euler characteristic of $\mathsf{K}$}.
By the Euler--Poincare formula (see Equation \eqref{eq:chi-abs}, c.f. \cite{Hatcher02}) we have the following relation between \eqref{eq:rnd-num-faces-def} and \eqref{eq:rnd-euler-char-def}
\begin{equation}\label{eq:rnd-euler-char-def-f_d}
 \tChi=\sum^{n-1}_{d=0} (-1)^{d} \pmb f_d.
\end{equation}
Moreover,
\[
 \tChi(\ts)(0)=\chi(\ts)=r^+(\ts)-r^-(\ts).
\]
\begin{proposition}
We have the following formulas for the coefficients of $\pmb f_d$ and $\tChi$: 
\begin{equation}\label{eq:c-F-simple}
c_{\ts,k}(\pmb f_d)=\sum^{r_{top}(\ts)}_{i=1} (-1)^{r_{top}(\ts)-i} {r_{top}(\ts) \choose i} i^k,
\end{equation}
\begin{equation}\label{eq:c-euler-simple}
\begin{split}
c_{\ts,k}(\mathcal{\tChi}) & =\sum_{\tl\in \mathfrak{C}_n;\tl\subseteq \ts} (-1)^{|\ts|-|\tl|} (\tChi(\tl)(0))^k=\sum_{\tl\in \mathfrak{C}_n;\tl\subseteq \ts} (-1)^{|\ts|-|\tl|}\bigl(r^+(\tl)-r^-(\tl)\bigr)^k\\
& =\sum^{r^+_{top}(\ts)}_{i=0}\sum^{r^-_{top}(\ts)}_{j=0} (-1)^{r_{top}(\ts)-i-j}{r^+_{top}(\ts) \choose i}{r^-_{top}(\ts) \choose j}\bigl(i-j+r^+_{low}(\ts)-r^-_{low}(\ts)\bigr)^{k}\\
\end{split}
\end{equation}
\end{proposition}
\begin{proof}[Proof of Formula \eqref{eq:c-F-simple}] Applying \eqref{eq:c_J,k} directly to $\pmb f_d$ we obtain the first identity in \eqref{eq:c-F-simple}. For the second equation in \eqref{eq:c-F-simple}, let $\tl\in \mathfrak{P}_n$ be the set of all $d$--faces. Since
$\pmb f_d=\sum_{I\in \tl} e_I$, for any $\tk\subseteq \tl$, Equation  \eqref{eq:C_J-eval}
implies
\begin{equation}\label{eq:c_k-f_d}
 c_{\tk}((\pmb f_d)^k)=\sum_{\tr\in \mathfrak{P}_n;\tr\subseteq \tk} (-1)^{|\tk|-|\tr|}\, (\pmb f_d(\tr)(0))^k=\sum^{|\tk|}_{i=1} (-1)^{|\tk|-i} {|\tk| \choose i} i^k. 
\end{equation}
Considering $\pmb f_d$ as an element of $\R_{\mathcal{I}}[e_I]$ and choosing a chain representative for $\pmb f_d^k$, we conclude that its coefficients $c_{\ts,k}(\pmb f_d^k)$ vanish 
unless the corresponding antichain $\widehat{\ts}$ consists of purely $d$--faces. In the latter case 
we obtain from \eqref{eq:c_k-f_d}
\[
 c_{\ts,k}((\pmb f_d)^k)=c_{\tk}((\pmb f_d)^k),\qquad \text{for}\quad \tk=\widehat{\ts},
\]
which implies the identity in \eqref{eq:c-F-simple} via the notation of \eqref{eq:faces-notation}.
\end{proof}
\no Next, we turn to the random polynomial $\tChi=\tChi(\mathsf{K})$. The range of $\tChi(\mathsf{K})$ is contained in $[\underline{m}(\tChi),\overline{m}(\tChi)]$ where 
\begin{equation}\label{eq:range-chi}
  \underline{m}(\tChi)=-\sum_{r;0< 2 r+1\leq n} {n \choose 2 r+1},\qquad\text{and}\qquad
  \overline{m}(\tChi)=\sum_{r;0< 2 r\leq n} {n \choose 2 r}.
\end{equation}
If $\mathsf{K}$ is supported on some subcomplex $\tk\in \mathfrak{C}_n$, smaller than the full $n$--simplex,  the above range can be narrowed to 
\[
 \underline{m}({\tChi(\mathsf{K})})=-\sum_{0\leq 2 r+1\leq \dim(\tk)} {\pmb f_{2 r+1}}(\tk),\qquad \overline{m}({\tChi(\mathsf{K})})=\sum_{0\leq 2 r\leq \dim(\tk)} {\pmb f_{2 r}}(\tk).
\]
\begin{proof}[Proof of Formula \eqref{eq:c-euler-simple}] Applying \eqref{eq:c_J,k} to $\tQ=\tChi$ directly, one obtains the first part of \eqref{eq:c-euler-simple}. To obtain the second part  we choose to present a 
different argument for the purpose of cross verification. Recall that given indeterminates $x_1,\ldots,x_m$, we have the following multinomial formula (c.f. \cite{Feller71})
\begin{equation}\label{eq:multinomial}
 (x_1+x_2+\ldots+x_m)^k=\sum_{\substack{\boldsymbol{\alpha}=(\alpha_1,\alpha_2,\ldots,\alpha_m),\\
|\pmb{\alpha}|=k}} {k \choose \boldsymbol\alpha}\, x_1^{\alpha_1}
x_2^{\alpha_2}\ldots x_m^{\alpha_m},
\end{equation}
\no where ${k \choose \boldsymbol{\alpha}}=\frac{k!}{\alpha_1 !\alpha_2 !\ldots \alpha_m !}$,  $\alpha_i\geq 0$, $|\pmb{\alpha}|=\sum_i\alpha_i$ and $\boldsymbol{\alpha}$ form all possible partitions of $k$. Let $\pmb\alpha$ have  coordinates indexed by $f(n)$ (i.e. faces of $\mathbf{\Delta}_n$). A direct application of \eqref{eq:multinomial} to \eqref{eq:rnd-euler-poly} yields
\begin{equation}\label{eq:chi^k-multi}
\begin{split}
 (\tChi)^k & =\sum_{\substack{\pmb{\alpha}=(\alpha_I),\\
 |\pmb{\alpha}|=k}} {k \choose \pmb{\alpha}} \prod_{I\in f(n)}  \Bigl((-1)^{|I|-1} e_I\Bigr)^{\alpha_I}\\  & =\sum_{\substack{\pmb{\alpha}=(\alpha_I),\\ |\pmb{\alpha}|=k}} \Bigl((-1)^{\sum_{I\in \ts(\boldsymbol{\alpha})} (|I|-1)\alpha_I} \Bigr){k \choose \boldsymbol{\alpha}} e_{\ts(\boldsymbol{\alpha})},
\end{split}
\end{equation}
where we denoted
\begin{equation}\label{eq:J-alpha}
\ts(\pmb{\alpha})=\{I\in f(n)\ |\ \alpha_I> 0\}.
\end{equation}
Observe that for any $\pmb\alpha$ and $\pmb\alpha'$, 
\begin{equation}\label{eq:e-alpha=e-alpha'}
 e_{\ts(\pmb\alpha)}=e_{\ts(\pmb\alpha')},\qquad \text{in}\quad \R_{\mathcal I}[e_I],
\end{equation}
 if and only if the corresponding antichains are the same i.e. $\widehat{\ts(\pmb\alpha)}=\widehat{\ts(\pmb\alpha')}$. Fix a chain representative of some complex $\ts\in \mathfrak{C}_n$ and let $\widehat{\ts}$ be the corresponding antichain. Clearly, $\widehat{\ts}\subseteq \ts$, consider  partitions $\pmb\alpha$ of $k$ which are in the form $\pmb\alpha=\pmb\beta+\pmb\gamma$ where $\pmb\beta=(\beta_I)$, satisfies: $\beta_I>0$ for $I\in \widehat{\ts}$ and $\beta_I=0$ for $I\in\ts-\widehat{\ts}$, and   $\pmb\gamma=(\gamma_I)$ satisfies:  $\gamma_I\geq 0$ for $I\in \ts-\widehat{\ts}$ and $\gamma_I=0$ for $I\in \widehat{\ts}$. The following claim immediately follows

\no {\bf Claim:} Given $\ts\in \mathfrak{C}_n$ and any partition $\pmb\alpha$ of $k$ indexed by $f(n)$,  we have  $\widetilde{\Pi}(\ts(\pmb\alpha))=\ts$ if and only if $\pmb\alpha$ has the above decomposition: $\pmb\beta+\pmb\gamma$.

\no Therefore, the $c_{\ts,k}(\mathcal{\tChi})$ coefficient of the chain representative of $(\tChi)^k$ is a sum of coefficients of $e_{\ts(\pmb\alpha)}$ for all $\pmb\alpha$ in the form $\pmb\beta+\pmb\gamma$. Applying notation \eqref{eq:faces-notation} we may express it as
\begin{align}\label{eq:ch-exp-c_J}
 & \hspace{2cm} (\tChi)^k =\sum_{\ts\in \mathfrak{C}_n} c_{\ts,k}(\tChi)\, e_{\ts},\qquad\mbox{where}\\
\notag c_{\ts,k}(\mathcal{\tChi}) & =\begin{cases}
 {\displaystyle \sum_{\substack{(\pmb\beta,\pmb\gamma)=(\beta_1,\ldots,\beta_{r_{top}},\gamma_1,\dots,\gamma_{r_{low}}),\\ 
|\pmb\beta|+|\pmb\gamma|=k,\,\beta_i>0,\gamma_j\geq 0}}} {\displaystyle (-1)^{\sum^{r_{top}}_{i=1} (|I_i|-1)\beta_i+\sum^{r_{low}}_{j=1} (|J_j|-1)\gamma_j}{k \choose \pmb\beta,\pmb\gamma}}, & \quad \text{if}\ k\geq r,\\
 0, & \quad \text{otherwise},
\end{cases}
\end{align}
where we indexed the faces of $\widehat{\ts}$ in $\ts$ by $\{I_i\}$, $i=1,\ldots, r_{top}$ and faces of $\ts-\widehat{\ts}$ in $\ts$  by $\{J_j\}$, $j=1,\ldots, r_{low}$. To set up the inclusion--exclusion principle, note that the sum for $c_{\ts,k}(\mathcal{\tChi})$ is a part of the 
larger sum (where we allow $\beta_i\geq 0$, and $(\pmb\beta,\pmb\gamma)=(\beta_1,\ldots,\beta_{r_{top}},\gamma_1,\dots,\gamma_{r_{low}})$):
\[
{\displaystyle \sum_{\substack{(\pmb\beta,\pmb\gamma)\\ 
|\pmb\beta|+|\pmb\gamma|=k,\,\beta_i\geq 0,\gamma_j\geq 0}}} {\displaystyle (-1)^{\sum^{r_{top}}_{i=1} (|I_i|-1)\beta_i+\sum^{r_{low}}_{j=1} (|J_j|-1)\gamma_j}{k \choose \pmb\beta,\pmb\gamma}}=\Bigl(\sum^{r_{low}}_{i=1} (-1)^{(|I_i|-1)}+\sum^{r_{top}}_{j=1} (-1)^{(|J_j|-1)}\Bigr)^k.
\]
We stratify the above sum with respect to number of $\beta_i$'s strictly greater than zero, and set up the inclusion--exclusion as follows. Let $F=\{1,\ldots,r_{top}\}$ and define 
for any $S\subseteq F$, functions $f$, $g$ (in \eqref{eq:incl-excl-g-to-f}, \eqref{eq:incl-excl-f-to-g}) as
\[
 \begin{split}
  f(S) & = {\displaystyle \sum_{\substack{(\pmb\beta,\pmb\gamma)=(\{\beta_i\},\{\gamma_j\}), 
|\pmb\beta|+|\pmb\gamma|=k,\gamma_j\geq 0,\\ 
\beta_i> 0,\ \text{if $i\in S$},\ \beta_i=0\ \text{if $i\not\in S$.}}}} {\displaystyle (-1)^{\sum^{r_{top}}_{i=1} (|I_i|-1)\beta_i+\sum^{r_{low}}_{j=1} (|J_j|-1)\gamma_j}{k \choose \pmb\beta,\pmb\gamma}},\\
 g(S) & = \Bigl(\sum_{i\in S} (-1)^{(|I_i|-1)}+\sum^{r_{low}}_{j=1} (-1)^{(|J_j|-1)}\Bigr)^k.
 \end{split}
\]
Observe that $\sum^{r_{low}}_{j=1} (-1)^{(|J_j|-1)}=r^+_{low}-r^-_{low}$, which yields
\[
g(S)=\Bigl(|S^+|-|S^-|+r^+_{low}-r^-_{low}\Bigr)^k.
\]
where $|S^+|$($|S^-|$) denotes number of even(odd) dimensional faces of $\widehat{\ts}$ indexed by $S$. 
By \eqref{eq:incl-excl-f-to-g} we obtain
\[
 f(F)=\sum_{S:S\subseteq F} (-1)^{r_{top}-|S|} \Bigl(|S^+|-|S^-|+r^+_{low}-r^-_{low}\Bigr)^k.
\]
Since there are $r^+_{top}$ even dimensional faces and $r^-_{top}$ odd dimensional faces in $\widehat{\ts}$, for a fixed $i\in [0,r^+_{top}]$ and $j\in [0,r^-_{top}]$ there are
exactly ${r^+_{top} \choose i}{r^-_{top} \choose j}$ subsets $S\subseteq F$ satisfying $i=|S^+|$, $j=|S^-|$. Thus the second part of \eqref{eq:c-euler-simple} now follows from  $f(F)=c_{\ts,k}(\mathcal{\tChi})$.
\end{proof}

As the last case of interest, we consider is the relative Euler characteristic $\tChi_{rel}=\tChi_{rel}(\mathsf{K}, \mathsf{L})$ of a random pair $(\mathsf{K},\mathsf{L})$. Denoting the characteristic functions of $\mathsf{K}$ by $\{e_I\}$ and of $\mathsf{L}$ by $\{w_J\}$, 
\eqref{eq:chi-rel} and \eqref{eq:rk(X,Y)} imply the following polynomial expression 
\begin{equation}\label{eq:rnd-euler-char-rel}
\begin{split}
   \tChi_{rel} & =\sum^{n-1}_{d=0}(-1)^{k}\Bigl(\sum_{I\in f_d(n)} (e_{I}- w_{I})\Bigr).
\end{split}
\end{equation}
\no Analogously, as in the absolute case, the distribution of $(\mathsf{K},\mathsf{L})$ is determined by 
\begin{equation}\label{eq:p_IJ-e-f}
p_{\ts,\tr}=\Pm(e_{\ts}=1,w_{\tr}=1)=\Pm(e_{\ts}w_{\tr}=1).
\end{equation}
\no The maximal constants for the range of  $\tChi_{rel}(\mathsf{K},\mathsf{L})$ are
\begin{equation}\label{eq:euler-range-rel}
 \underline{m}(\tChi_{rel})=\underline{m}(\tChi)-\overline{m}(\tChi),\quad \text{and}\quad \overline{m}(\tChi_{rel})=\overline{m}(\tChi)-\underline{m}(\tChi).
\end{equation}
\no For convenience we state the following corollary of Theorem \ref{thm:Q-dist}:

\begin{corollary}[Distribution of $\tChi_{rel}(\mathsf{K},\mathsf{L})$]\label{cor:euler-dist-rel}
 Given a random pair $(\mathsf{K},\mathsf{L})$, the distribution of $\tChi_{rel}$ on $[\underline{m}(\tChi_{rel}), \overline{m}(\tChi_{rel})]$ is 
 given as follows, for $j\in [0, N]$, $N=\overline{m}(\tChi_{rel})-\underline{m}(\tChi_{rel})$  
\begin{align}\label{eq:euler-dist-rel}
\Pm\bigl(\tChi_{rel}=\underline{m}(\tChi_{rel})+j\bigr)& =\sum_{(\ts,\tr)\in \mathfrak{C}_n\times \mathfrak{C}_n} a_{\ts,\tr,j}(\tChi_{rel})\,p_{\ts,\tr},\\
\notag a_{\ts,\tr,j}(\tChi_{rel}) & =\sum_{k=0}^{N}\bigl(v_{k j}(\tChi_{rel})\,c_{\ts,\tr, k}(\tChi_{rel})\bigr),
\end{align} 
where (using the notation of \eqref{eq:faces-notation})
\begin{align}\label{eq:c-euler-rel-simple}
&\mathbb{E}((\tChi_{rel}(\mathsf{K},\mathsf{L}))^k)  =\sum_{(\ts,\tr)\in \mathfrak{C}_n\times \mathfrak{C}_n} c_{\ts,\tr,k}\,p_{\ts,\tr},\qquad c_{\ts,\tr,k}=c_{\ts,\tr,k}(\tChi_{rel})\\
& \notag c_{\ts,\tr,k}=\begin{cases}
\displaystyle\sum_{\substack{{\scriptstyle i\in [0,r^+_{top}(\ts)],j\in [0,r^-_{top}(\ts)]},\\ {\scriptstyle i'\in [0,r^+_{top}(\tr)],j'\in [0,r^-_{top}(\tr)]}}}(-1)^{{\scriptstyle r_{top}(\ts)+r_{top}(\tr)-i-j-i'-j'}}
{{\scriptstyle r^+_{top}(\ts)} \choose {\scriptstyle i}}{{\scriptstyle r^-_{top}(\ts)} \choose {\scriptstyle j}}{{\scriptstyle r^+_{top}(\tr)} \choose {\scriptstyle i'}}{{\scriptstyle r^-_{top}(\tr)} \choose {\scriptstyle j'}}\cdot\\
\hfill\cdot{\scriptstyle \bigl((i-j)+(i'-j')+(r^+_{low}(\ts)-r^-_{low}(\ts))+(r^+_{low}(\tr)-r^-_{low}(\tr))\bigr)^{k}},\ \text{for}\quad k\geq r,\\
\\
\qquad\qquad 0,\qquad \text{for}\quad k<r.
\end{cases}
\end{align} 
\end{corollary}
\no The proof is as fully analogous the previous arguments and is omitted. Note that the expression for 
$c_{\ts,\tr,k}(\tChi_{rel})$ in  \eqref{eq:c-euler-rel-simple} simplifies to \eqref{eq:c-euler-simple} whenever $\mathsf{L}=\text{\rm \O}$.


\section{Coverings of one--complexes and the Euler characteristic.}\label{sec:topological}
Given  a deterministic {\it covering} of a finite simplicial complex $X$, i.e. a collection of compact connected subsets  $A=\{A_{\{i\}}\}$, we can define its  {\em nerve}, $\mathcal{N}(A)$  as a finite
 complex where vertices $\{i\}$ are just elements $A_{\{i\}}$ of the covering and a $k$-face $I=\{i_1,\ldots, i_{k+1}\}$ belongs to $\mathcal{N}(A)$, if and only if $A_{\{i_1\}}\cap A_{\{i_2\}}\cap\ldots\cap A_{\{i_{k+1}\}}\neq \text{\rm \O}$ (c.f. \cite{Wallace70}). 
\begin{figure}[ht]
\vspace{-.5cm}
\begin{center}
   \includegraphics[width=.45\textwidth,height=.3\textheight]{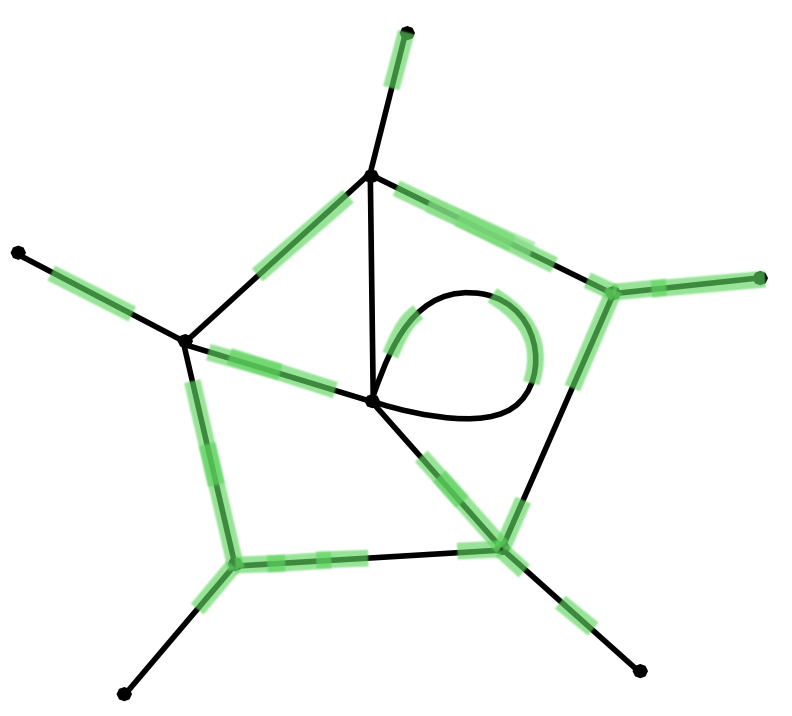}
\end{center}
{\footnotesize Figure \ref{fig:1-complex}: An example of a 1-complex with marked realization of a good cover.}
	\label{fig:1-complex}
\end{figure}
  
 The following result, due to Borsuk \cite{Borsuk48}, is of fundamental importance in algebraic topology
\begin{lemma}[The Nerve Lemma \cite{Borsuk48}]\label{lem:nerve}
 Let $A=\{A_{\{i\}}\}$ be a covering of $X$ and $\mathcal{N}(A)$ the associated nerve. If all intersections $A_{\{i_1\}}\cap A_{\{i_2\}}\cap\ldots\cap A_{\{i_{k+1}\}}$, for $k>0$ are contractible, then
$\mathcal{N}(A)$ has a homotopy type of the subspace $|A|=\bigcup_i A_{\{i\}}$ of $X$.
\end{lemma}
 Recall that a subset of $X$ is contractible 
if it can be deformed continuously to a point  \cite{Hatcher02}. 
If $A=\{A_{\{i\}}\}$ satisfies the assumption of this lemma then we call it a {\em good} covering (of $X$).

 In the remainder of this section we collect 
elementary facts from algebraic topology and show how the Euler characteristic of $\mathcal{N}(A)$ provides a criteria for a good deterministic covering $A=\{A_{\{i\}}\}$, to completely cover a connected $1$--complex $X$, the proofs are basic and are either omitted or deferred to Appendix \ref{sec:appendix}. 

\subsection{Coverage and the nerve complex}\label{sec:cov-via-nerve} We assume throughout that $X$ is a connected $1$--complex (c.f. \cite[p. 103]{Hatcher02}) homeomorphic to a multi-graph, and denote $\partial X$ the set of leaf vertices of $X$. 

\begin{proposition}\label{prop:chi-ineq}
 Let $\{A_{\{i\}}\}$ be a good covering  of $X$, $|A| = \bigcup_{i} A_{\{i\}}$, denote  $U = |A|$ and $V = \overline{|A|^c}$. Then, 
\begin{equation}\label{eq:betti_1X-U}
 \beta_1(X) \ge \beta_1(U),
\end{equation}
and 
\begin{equation}\label{eq:chiX-U}
 \chi(X) \le \chi(U).
\end{equation}
Moreover, if the inequality in \eqref{eq:betti_1X-U} is strict then \eqref{eq:chiX-U} is also strict.
\end{proposition}
\no By the Nerve Lemma, an obvious necessary condition for $X\subseteq |A|$ is
\begin{equation}\label{eq:chiX=chiA}
 \chi(X)=\chi(|A|)=\chi(\mathcal{N}(A)).
\end{equation}
\no If $\partial X=\text{\rm \O}$, we have the following 
\begin{corollary}\label{cor:chi-covers-no-bdry}
 Suppose $X$ satisfies $\partial X=\text{\rm \O}$, then \eqref{eq:chiX=chiA} implies $X\subseteq |A|$.
\end{corollary}

\no When $\partial X\neq \text{\rm \O}$, the condition \eqref{eq:chiX=chiA} is insufficient; however we may adjust it by using the relative version $\chi_{rel}(X,\partial X)$ of the Euler characteristic \eqref{eq:chi-rel}. Note that for the pair $(X,\partial X)$, $\chi_{rel}(X,\partial X)$ reduces to 
\[
\chi_{rel}(X,\partial X)=\chi(X)-\#\{\partial X\},
\]
where $\#\{\partial X\}$ is a number of points in $\partial X$. By \cite[p. 102]{Hatcher02} we may consider the quotient complex $X'=X/\partial X$ which is a $1$--complex (\cite[p. 103]{Hatcher02}) with $\partial X'=\text{\rm \O}$, and
\[
 \chi_{rel}(X,\partial X)=\chi(X/\partial X).
\]
 Let $q:X\mapsto X'$ be the quotient projection, then the covering $A$ of $X$ projects to the covering $A'$ of $X'$. It is not true that $A'$ is automatically a good covering of $X'$, one may easily find examples where this is the case. However, the following  
 fact is available (proof left to the reader)
\begin{lemma}\label{lem:obvious}
 Given $A=\{A_{\{i\}}\}$ is a good covering of $X$, let for every $i$ the intersection $A_{\{i\}}\cap \partial X$  
be either empty or a point (in other words $A_{\partial X}=\{A_{\{i\}}\cap \partial X\}$ is a good covering of $\partial X$). Then the quotient covering $A'$ of $X'$ is also good.  
\end{lemma}
\no Consequently, we say that $A$ is a {\em good covering of the pair} $(X,\partial X)$ provided $A$ is good for $X$ and $A_{\partial X}$ is good for $\partial X$. Then by the above lemma $A'$ is good for $X'$ and Corollary  \ref{cor:chi-covers-no-bdry} says that 
$A'$ covers $X'$, if and only if $\chi(|A'|)=\chi(X')$.  It leads us to the following generalization of Corollary \ref{cor:chi-covers-no-bdry}.
\begin{lemma}\label{lem:chi-covers-bdry}
 Given a good covering $A=\{A_{\{i\}}\}$ of $(X,\partial X)$  let $|A|=\bigcup_i A_{\{i\}}$. Then $X\subseteq |A|$,   if and only if 
\begin{equation}\label{eq:chiX=chiA-rel}
 \chi_{rel}(\mathcal{N}(A),\mathcal{N}(A_{\partial X}))=\chi_{rel}(X,\partial X)
\end{equation}
or equivalently 
\begin{equation}\label{eq:chiX=chiA-rel-partial}
 \chi(|A|)=\chi(X)-\#\{\partial X\}+\#\{|A|\cap \partial X\}.
\end{equation}
\end{lemma}

\begin{remark}
 {\rm 
 Equivalently, the coverage condition for $(X,\partial X)$ can be obtained by looking at the covering $\widehat{A}$, equal to a union of $A$ and the boundary vertices: 
$\partial X=\{x_1,\ldots,x_{\#\{\partial X\}}\}$. Then $\widehat{A}$ is good if satisfies the conditions of Lemma \ref{lem:obvious} 
\begin{align*}
 \chi(|\widehat{A}|) & =\chi(|A|\cup \partial X)=\chi(|A|)+\chi(\partial X)-\chi(|A|\cap \partial X)\\
& =\chi(|A|)+\#\{\partial X\}-\#\{|A|\cap \partial X\},
\end{align*}
which together with \eqref{eq:chiX=chiA} leads us to \eqref{eq:chiX=chiA-rel}. 
}
\end{remark}

\subsection{Coverage of \texorpdfstring{$X$}{X} by \texorpdfstring{$\varepsilon$}{epsilon}-balls. Vietoris--Rips complex.}\label{sec:nerve-A-rips}

A special case of interest  (see e.g. \cite{Wendl07, deSilva-Ghrist07b})  is when a connected $1$--complex $X$ ought to be covered by $\varepsilon$-size neighborhoods, and $\varepsilon$ can be sufficiently small. In such cases the topology of $\mathcal{N}(A)$  
simplifies and one may work with  Vietoris--Rips complex \cite{Hausmann95}, as we show in the following paragraphs. 

\no Recall that given a simplicial complex $K$ its {\em Vietoris--Rips complex} $\mathcal{R}(K)$, \cite{Hausmann95} is defined to be a maximal simplicial complex (with respect to inclusion) which 
has the same $1$-skeleton as $K$. In  practice, this means that $\mathcal{R}(K)$ is obtained by filling every $k$-clique in the graph $K^{(1)}$ with a $(k-1)$-dimensional face, e.g. $3$-cycles are filled with $2$-simplices in $\mathcal{R}(K)$, etc.

 We will consider a finite covering $A=\{A_{\{1\}},\ldots, A_{\{n\}}\}$ of $(X,d_X)$ by closed $\varepsilon$-balls. Possible shapes of such balls for  $\varepsilon$ sufficiently small are depicted on Figure \ref{fig:e-balls}. 
%
\begin{figure}[htbp]
    \begin{center}
     \includegraphics[width=0.9\textwidth,height=0.17\textheight]{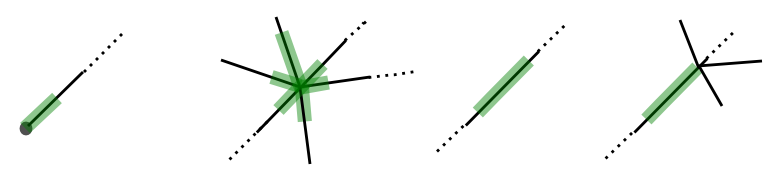}   
\end{center}
     { Figure \ref{fig:e-balls}: Possible shapes of closed $\varepsilon$-balls in $X$ with the intrinsic distance $d_X$.}\label{fig:e-balls}
\end{figure}


\no Let us denote by $\mathcal{R}(A)$ the Vietoris--Rips complex of the nerve of the cover, and record the following 
\begin{proposition}\label{prop:good-rips-eps-cover}
 Suppose $\mathcal{C}$ is the girth of $X'$, i.e. the length of the shortest cycle in the quotient complex $X'=X/\partial X$. Then, 
\begin{itemize}
\item[$(i)$]  if $\varepsilon < \frac{1}{4}\mathcal{C}$, the covering $A$ by $\varepsilon$-balls in $(X,d_X)$ is a good cover. 
\item[$(ii)$] if $\varepsilon < \frac{1}{6}\mathcal{C}$, the nerve $\mathcal{N}(A)$ of $A$   equals   $\mathcal{R}(A)$.
\end{itemize}
\end{proposition}
\begin{proof}
 For $(i)$ we must show that every $k$-fold intersection $A_{\{i_1\}}\cap A_{\{i_2\}}\cap\ldots \cap A_{\{i_k\}}$ has a homotopy type of a point. Because $\text{diam}(A_{\{i\}})< \mathcal{C}$, $A_{\{i\}}$ is a connected tree and therefore contractible, which shows the claim for $k=1$.  For $k=2$, first suppose that a nonempty intersection $A_{\{i\}}\cap A_{\{j\}}$ is disconnected i.e. $\dim(\widetilde{H}_0(A_{\{i\}}\cap A_{\{j\}}))\geq 1$ (where $\widetilde{H}_\ast(\,\cdot\,)$ denotes the reduced homology groups c.f. \cite{Hatcher02}). Since $A_{\{i\}}$ and $A_{\{j\}}$ are connected, the reduced Mayer-–Vietoris  sequence  for $A_{\{i\}}\cap A_{\{j\}}$ then simplifies to
\[
 0\longrightarrow \widetilde{H}_1(A_{\{i\}}\cup A_{\{j\}})\longrightarrow \widetilde{H}_0(A_{\{i\}}\cap A_{\{j\}})\longrightarrow \widetilde{H}_0(A_{\{i\}})\oplus\widetilde{H}_0(A_{\{j\}})\cong \{0\},
\]
We obtain  $\widetilde{H}_1(A_{\{i\}}\cup A_{\{j\}})\cong \widetilde{H}_0(A_{\{i\}}\cap A_{\{j\}})\cong \R^k$ for some $k\geq 1$, which implies that $A_{\{i\}}\cup A_{\{j\}}$ contains a nontrivial cycle. This however contradicts the fact that $\text{diam}(A_{\{i\}}\cup A_{\{j\}})\leq 4\,\varepsilon< \mathcal{C}$.  Thus $k$ has to vanish and $A_{\{i\}}\cap A_{\{j\}}$ must be connected, contain no cycle, and is therefore  contractible. Now, for an induction 
step with respect to $k$, it suffices to apply the previous step to  $A'=A_{\{i_1\}}\cap A_{\{i_2\}}\cap\ldots \cap A_{\{i_k\}}$ and $A''=A_{\{i_{k+1}\}}$.

Before proving $(ii)$, recall the $1$-dimensional version of Helly's Theorem (c.f. \cite{Eckhoff93}) implies that given a finite collection 
of intervals $\{C_1,C_2,\ldots,C_n\}$ on $\R$, if the intersection of each pair is nonempty, i.e. $C_i\cap C_j\neq \text{\rm \O}$, for every $1\leq i,j\leq n$, then $\bigcap^n_{i=1} C_i\neq \text{\rm \O}$.

First consider the case of  $3$-fold intersections, i.e. supposing that $A_{\{j\}}\cap A_{\{k\}}\neq \text{\rm \O}$, $1\leq k\neq j\leq 3$. We aim to show that $A_{\{1\}}\cap A_{\{2\}}\cap A_{\{3\}}\neq \text{\rm \O}$. Observe that $V=A_{\{1\}}\cup A_{\{2\}}\cup A_{\{3\}}$ is connected and by the argument of $(i)$ it must be a connected tree, i.e. contains no cycles. Let $p_{1,2}$, $p_{2,3}$, $p_{1,3}$ be distinct points in $V$ such that $p_{i,j}\in A_{\{i\}}\cap A_{\{j\}}$. 
Note that for each pair: $p_{i,j}$, $p_{s,t}$ there exists a path in $V$ connecting these points. We now consider two cases: (1) one of these paths, we denote by $l$, contains all three points $p_{i,j}$, then  the collection $\{C_i\}$, $C_i=l\cap A_{\{i\}}$, $i=1,2,3$ satisfies the assumptions 
of Helly's Theorem which implies the claim. (2) none of the paths between paris of  $p_{i,j}$'s contain the third point. Consider two shortest paths: 
$l_1$ between $p_{1,2}$ and $p_{2,3}$, and $l_2$  between $p_{1,2}$ and $p_{2,3}$ then $l_{1,2}=l_1\cap l_2$ is a segment between $p_{1,2}$ and some vertex of $v\in V$. 
The vertex $v$ has to be in one of $A_{\{j\}}$'s, w.l.o.g. suppose $v\in A_{\{2\}}$ (as other cases are analogous.) Then if $v$ is also in $A_{\{1\}}$ or $A_{\{3\}}$ we can take $p_{1,2}$ or $p_{2,3}$ equal to 
$v$ and use (1). If $v\notin A_{\{1\}}$ and $v\notin A_{\{3\}}$ then we observe that either $A_{\{1\}}$ or $A_{\{3\}}$ is disconnected which is not the case. This concludes the proof of $(ii)$ for the $3$-fold case, the general case can be obtained by induction.
\end{proof}

\section{Complete coverage probability.}\label{sec:complete-cover-prob}

\no In this section we interpret results of Sections \ref{sec:cov-via-nerve}--\ref{sec:nerve-A-rips} in the random setting.

\subsection{Random coverings and the random nerve}\label{sec:rand-nerve}  
 
Suppose $\mathsf{A}=\{\mathsf{A}_{\{i\}}\}$ is a random covering of a metric space $X$. We define the {\em nerve} $\mathcal{N}(\mathsf{A})$ of $\mathsf{A}$ by defining a probability measure $\mathbb{P}_{\mathsf{A}}$ on $\mathfrak{C}_n$  via the process elucidated in Section \ref{sec:introduction} in \eqref{eq:p_I} and \eqref{eq:p_s}. Observe that given a subspace $Y\subseteq X$ we obtain an induced random covering $\mathsf{A}_Y$ from $\mathsf{A}$:
\[
 \mathsf{A}_Y=\{\mathsf{A}_{\{1\}}\cap Y,\mathsf{A}_{\{2\}}\cap Y,\ldots, \mathsf{A}_{\{n\}}\cap Y\}
\]
The definition of $\Pm_{\mathsf{A}}$ extends to pairs $(\mathcal{N}(\mathsf{A}),\mathcal{N}(\mathsf{A}_Y))$ in an obvious way. In particular given $(\ts,\tr)\in\mathfrak{C}_n\times\mathfrak{C}_n$, we set 
\begin{equation}\label{eq:p_sr}
\begin{split}
 p_{\ts,\tr} & =\Pm(\{(\tk,\tl)\in \mathfrak{C}_n\times\mathfrak{C}_n \ |\ \ts\subseteq \tk,\tr\subseteq \tl\})\\
 & =\Pm\bigr(\forall_{I\in \ts}\bigl\{\bigcap_{i\in I} \mathsf{A}_{\{i\}}\neq \text{\rm \O}\bigr\}\,, \forall_{\{J\}\in \tr}\bigl\{\bigcap_{j\in J} \mathsf{A}_{\{j\}}\cap Y\neq \text{\rm \O}\bigr\}\,\bigl).
\end{split}
\end{equation}
Clearly, $\NA$ is a random complex, and $(\NA,\mathcal{N}(\mathsf{A}_Y))$ is a random pair.
We say a finite random covering $\{\mathsf{A}_{\{i\}}\}_{i=1,\ldots,n}$ of $X$ is {\em good} if and only if it is a good covering on $X$ almost surely. Further, we say a random covering $\mathsf{A}=\{\mathsf{A}_{\{i\}}\}$ of a pair $(X,\partial X)$ is good provided it is a good covering of $X$ and $\mathsf{A}_{\partial X}$ is a good covering of $\partial X$.  $|\mathsf{A}|$ will denote the random set $\bigcup_i \mathsf{A}_{\{i\}}$. 

\subsection{Proof of the extended version of Theorem \ref{thm:euler-coverage}} Let $\tChi_{rel}(\mathsf{A},\mathsf{A}_{\partial X})$ be the relative Euler characteristic of the pair $(\mathcal{N}(\mathsf{A}),\mathcal{N}(\mathsf{A}_{\partial X}))$.
We may now state Theorem \ref{thm:euler-coverage} for a general 1--complex $X$. 
\begin{theorem}[Coverage probability of a 1-complex $X$ with $\partial X\neq \text{\rm \O}$]\label{thm:euler-coverage-rel}
 Let $\mathsf{A}=\{\mathsf{A}_{\{i\}}\}$, $i=1,\ldots,n$ be a random good covering of the pair $(X,\partial X)$. Then, the range of $\tChi_{rel}(\mathsf{A},\mathsf{A}_{\partial X})$ can be restricted to 
\begin{equation}\label{eq:range-chi-A-rel}
 \underline{m}=\chi_{rel}(X,\partial X)\leq \tChi_{rel}(\mathsf{A},\mathsf{A}_{\partial X})\leq n=\overline{m},
\end{equation}
and the complete coverage probability equals
\begin{equation}\label{eq:cover-prob-euler-rel}
\begin{split}
 \Pm(X\subseteq |\mathsf{A}|) & =\Pm\bigl(\tChi_{rel}(\mathsf{A},\mathsf{A}_{\partial X})=\chi_{rel}(X,\partial X)\bigr),\\
 & =\sum_{(\ts,\tr)\in \mathfrak{C}_n\times\mathfrak{C}_n} a_{\ts,\tr}(\tChi_{rel})\,p_{\ts,\tr},
\end{split}
\end{equation}
where $a_{\ts,\tr}(\tChi_{rel})=a_{\ts,\tr,0}(\tChi_{rel})$ are defined in \eqref{eq:euler-dist-rel} of Corollary \ref{cor:euler-dist-rel}, and $p_{\ts,\tr}$ in \eqref{eq:p_sr}.
\end{theorem}
\begin{proof}
Under the given assumptions,  Lemma \ref{lem:chi-covers-bdry}  implies
\begin{equation}\label{eq:coverage=chi-rel}
\Pm(X\subseteq |\mathsf{A}|)=\Pm\bigl(\tChi_{rel}(\mathsf{A},\mathsf{A}_{\partial X})=\chi_{rel}(X,\partial X)\bigr).
\end{equation}
At this point the formula \eqref{eq:euler-dist-rel} of Corollary \ref{cor:euler-dist-rel} can be applied to the random pair $(\NA$, $\mathcal{N}(\mathsf{A}_{\partial X}))$ to give an exact expression 
for  $\Pm\bigl(\tChi_{rel}(\mathsf{A},\mathsf{A}_{\partial X})=\chi_{rel}(X,\partial X)\bigr)$. In this particular case the range of $\tChi(\mathsf{A},\mathsf{A}_{\partial X})$ is given by \eqref{eq:range-chi-A-rel}, 
where the lower bound follows from Proposition \ref{prop:chi-ineq}, and the upper bound corresponds to the case when elements of the covering $\mathsf{A}$ are pairwise disjoint and contained in $X-\partial X$, i.e. $\mathcal{N}(\mathsf{A})$ is just $n$ distinct points. 
The formula for $p_\ts$ in \eqref{eq:p_s-rips} is a direct consequence of Proposition \ref{prop:good-rips-eps-cover}, (see also Remark \ref{rem:prob-rips}).
\end{proof}

\begin{remark}
{\em
 Note that $\mathcal{N}(\mathsf{A}_{\partial X})$ generally contains high dimensional faces and therefore the chain expansion of $\tChi^k_{rel}$ in $\R_{\mathcal I}[e_I,w_J]$  involves monomials in $e_{\ts}$ and $w_{\tr}$. To simplify this expansion one may observe that $\mathcal{N}(\mathsf{A}_{\partial X})$ has a homotopy type of finitely many points or is empty. Specifically, from \eqref{eq:chiX=chiA-rel-partial} we have 
\[
 \tChi_{rel}(\mathsf{A},\mathsf{A}_{\partial X})=\tChi(\mathsf{A})-\#\{\mathsf{A}\cap \partial X\}.
\]
The random variable $\#\{\mathsf{A}\cap \partial X\}$ (counting points in $\mathsf{A}_{\partial X}$)
can be expressed as follows: 
\begin{equation}\label{eq:chi-rel-f}
 \tChi_{rel}(\mathsf{A},\mathsf{A}_{\partial X})=\tChi(\mathsf{A})-\sum^q_{i=1} w_{\{i\}}.
\end{equation}
where $\{1,\ldots,q\}$ label points of $\partial X$ and $\{w_{\{i\}}\}_{i=1,\ldots,q}$ are the indicator functions of points in  $\mathsf{A}_{\partial X}$. Consequently, we may derive expressions 
for powers $\tChi^k_{rel}$ as polynomials in  $\R[e_I,w_{\{i\}}]$. These  expansions of $\tChi^k_{rel}$ involve products of $e_{\ts}$ and $w_{\{i\}}$ only, which may provide a different way to express $\Pm(X\subseteq |A|)$.
}
\end{remark}

\begin{remark}\label{rem:prob-rips}
{\rm 
 In order to be more explicit about how the computation of $p_{\ts,\tr}$ simplifies in the case the nerve $\mathcal{N}(\mathsf{A})$ equals the Vietoris--Rips complex $\mathcal{R}(\mathsf{A})$, let us suppose $\mathsf{A}_{\{i\}}$ are $\varepsilon$-radius closed balls in $X$ with random centers $\xi_i\in X$. In $\mathcal{R}(\mathsf{A})$ any simplex indexed by $I=\{i_1,i_2,\ldots,i_k\}$ is determined by its edges, and an edge $\{i,j\}$ in $\mathcal{R}(\mathsf{A})$ occurs if and only if $|\xi_i-\xi_j|\leq 2\varepsilon$ (where $|\,\cdot\,-\,\cdot\,|$ is a short notation for the distance $d_X(\cdot,\cdot)$ on $X$). For instance, we have  
\[
 p_I=\Pm(\mathsf{A}_{\{i_1\}}\cap \mathsf{A}_{\{i_2\}}\cap \ldots \cap \mathsf{A}_{\{i_k\}}\neq \text{\rm \O})=\Pm\bigl(|\xi_{i_s}-\xi_{i_t}|\leq 2\varepsilon\ |\ \forall_{s,t}\ s\neq t\bigr).
\]
Enumerate points in $\partial X$ as follows $\{x_1,x_2,\ldots, x_{M}\}$, $M=\#\{\partial X\}$.  Now, $p_{\ts,\tr}$ given in \eqref{eq:p_sr} is just a volume of the set 
\[
  A_{\ts,\tr}=\{(\xi_1,\ldots,\xi_n)\in X^n\ |\ \forall_{I\in \ts}\forall_{\substack{s,t\in I,\\
s\neq t}}\ |\xi_s-\xi_t|\leq 2\varepsilon,\ \forall_{I\in \tr}\, \exists_{1\leq s\leq M}\ \forall_{i\in I}\ |\xi_i-x_s|\leq \varepsilon\},
\]
\no which in the case $\Pm=d\xi_1\,d\xi_2\ldots d\xi_n$ (i.e. $\xi_i$'s are independent) can be computed 
via ordinary calculus techniques or estimated numerically. These formulas further simplify, if $\partial X=\text{\rm \O}$, but we do not attempt these computations here.
}
\end{remark}

\section{Proof of Theorem \ref{thm:coverage-upperbounds}}\label{sec:estimates}

In this section we use the method of finite differences, c.f. \cite{Alon-Spencer-book08}, to give an upper bound for the complete coverage probability in terms of the expected Euler characteristic and prove Theorem \ref{thm:coverage-upperbounds}. Let $\{\mathsf{A}_{\{i\}}\}$, $i=1,\ldots,n$ be a finite good covering of $X$, consider the following shifted version of the relative Euler characteristic $\tChi_{rel}(\mathsf{A},\mathsf{A}_{\partial X})$ of $(\mathcal{N}(\mathsf{A}),\mathcal{N}(\mathsf{A}_{\partial X}))$:
\[
\tChi_0=\tChi_{rel}(\mathsf{A},\mathsf{A}_{\partial X})-\underline{m},
\]
where $\underline{m}=\chi_{rel}(X,\partial X)$. From \eqref{eq:chiX-betti} we obtain  
\begin{equation}\label{eq:chi-betti-1comp}
 \tChi_{rel}(\mathsf{A},\mathsf{A}_{\partial X})=\pmb\beta_0-\pmb\beta_1,
\end{equation}
where $\pmb\beta_\ast=\pmb\beta_\ast(\mathsf{A},\mathsf{A}_{\partial X})$ stand for the random relative Betti numbers. 
Recall that $\{e_I, w_J\}$, $I,J\in f(n)$ stand for the indicator functions of faces in $(\mathcal{N}(\mathsf{A}),\mathcal{N}(\mathsf{A}_{\partial X}))$. 

We will consider a filtration by random vectors $V_i$ denoting  
$(e_{I(i)},f_{J(i)})$ where  $I(i), J(i)\in f(n)$ are  subsets of  $\{1,\ldots, i\}$. Note that $V_i$ reveals subcomplexes in $\mathfrak{C}_n$ spanned by vertices $1$ through $i$. 
By analogy to the setting of Erd{\H{o}}s--R{\'e}nyi model   \cite{Alon-Spencer-book08}, we set up a {\em vertex exposure martingale}, associated with $\tChi_0$ and $\{V_i\}$ as follows: 
\begin{equation}\label{eq:simplex-martingale}
 \mathtt{Y}_{0}=\mu_0=E(\tChi_0),\qquad \mathtt{Y}_{i}=E(\tChi_0\ |\ V_i),\qquad i=1,\ldots,n.
\end{equation}
 Clearly, $\mathtt{Y}_n=\tChi_0$ and the sequence $\{\mathtt{Y}_i\}$ is an instance of Doob's martingale \cite{Alon-Spencer-book08}. Recall the following variant of the Azuma–-Hoeffding inequality \cite{Alon-Spencer-book08, Azuma67}, for $\{\mathtt{Y}_i\}$: 
\begin{equation}\label{eq:azuma}
 \Pm(\mathtt{Y}_{n}-\mathtt{Y}_0\leq -a)\leq \exp\Bigl(\frac{-a^2}{2\sum^{n}_{i=1} c^2_i}\Bigr)
\end{equation}
where $a>0$, and $c_i$ is a difference estimate
\begin{equation}\label{eq:diff-azuma}
 |\mathtt{Y}_i-\mathtt{Y}_{i-1}|\leq c_i.
\end{equation}
\no Exposing a vertex (or a face containing it) changes $\pmb\beta_0$ by at most $1$ and $\pmb\beta_1$
by at most $\pmb\beta_1(X,\partial X)=1-\chi_{rel}(X,\partial X)$ thus we obtain
\[
 |\mathtt{Y}_i-\mathtt{Y}_{i-1}|\leq 2+|\chi_{rel}(X,\partial X)|.
\]
Let $a=\mu_0$, then 
\[
\Pm(\tChi_0 =0)=\Pm(\tChi_0 \leq \mu_0-a)=\Pm(\mathtt{Y}_{n}-\mathtt{Y}_0\leq -a).
\]
Using the above estimates for $c_i$ and \eqref{eq:azuma} yields
\[
\Pm(X\subseteq |\mathsf{A}|)=\Pm(\tChi_0 =0)\leq \exp\Bigl(\frac{-\mu^2_0}{2n(|\chi_{rel}(X,\partial X)|+2)^2}\Bigr),
\]
which completes the proof of Theorem \ref{thm:coverage-upperbounds}.

\appendix
\section{Auxiliary proofs for Section \ref{sec:topological}}\label{sec:appendix}

\begin{proof}[Proof of Proposition \ref{prop:chi-ineq}]
\noindent Consider the Mayer-Vietoris sequence applied to $U$ and $V$:
\[
0 \rightarrow H_1(U\cap V) \stackrel{j_1}{\rightarrow} H_1(U) \oplus H_1(V) \rightarrow H_1(X) \rightarrow H_0(U \cap V) \rightarrow H_0(U) \oplus H_0(V) \rightarrow H_0(X) \rightarrow 0.
\]
Since $U\cap V=\partial A$ is just finitely many points, in real coefficients  we have 
\[
0 \longrightarrow \R^{\beta_1(U)} \oplus \R^{\beta_1(V)} \stackrel{d_1}{\longrightarrow} \R^{\beta_1(X)} \longrightarrow \ldots
\]
From \eqref{eq:chiX-betti}, $\chi(X) = 1 - \beta_1(X), \ \chi(U)= \beta_0(U) - \beta_1(U), \ \chi(V) = \beta_0(V) - \beta_1(V)$. Since $d_1$ is injective we have $\beta_1(U) + \beta_1(V) \le \beta_1(X)$, which  implies $-\beta_1(X) + \beta_1(U) \le 0$. This proves \eqref{eq:betti_1X-U}.

Now to prove \eqref{eq:chiX-U} we have two cases to consider: $\beta_0(U) > 1$ and $\beta_0(U) =1$. First assume $\beta_0(U) >1$. We argue by contradiction. That is, suppose $\chi(U) \le \chi(X)$. Then $\beta_0(U) -\beta_1(U) \le \beta_0(X) - \beta_1(X)$ so that $\beta_0(U) \le \beta_1(U) + 1 - \beta_1(X)$. But $\beta_1(A) - \beta_1(X) \le 0$ by the previous lemma. Therefore we obtain $\beta_0(U) \le 1$ contrary to our assumption. Now assume $\beta_0(U) =1$. Then $\chi(U) = 1 - \beta_1(U)$ and $\chi(X) = 1 - \beta_1(X)$ which yields $\chi(U) - \chi(X) = -\beta_1(U)  + \beta_1(X) \ge 0$. Thus $\chi(U) \ge \chi(X)$.
\end{proof}
\begin{proof}[Proof of Corollary \ref{cor:chi-covers-no-bdry}]
 Notice that generally $X$ (even with $\partial X\neq\text{\rm \O}$) is homotopy equivalent to a bouquet of circles.
  If $|A|^c\neq \text{\rm \O}$ in $X$, then (since $|A|^c$ is open) we pick $p\in |A|^c$ which is not a vertex of $X$. Then $p$ is in the interior of one of the edges which we denote by $e$. We may homotopy $X$ away from the interior of $e$ to a bouquet of  $r$ circles $S=\bigvee^r S^1$ in such a way that $p$ is away from the wedge point (just collapse along the edges different from $e$). From Proposition  \ref{prop:chi-ineq},
\[
 \beta_1(|A|)\leq \beta_1\bigl(\bigvee^{r-1} S^1\vee (S^1-\{p\})\bigr)<\beta_1(S)=\beta_1(X).
\]
Thus $\beta_1(|A|)<\beta_1(X)$ and therefore $\chi(X)<\chi(|A|)$, which implies the claim.
\end{proof}
\begin{proof}[Proof of Lemma \ref{lem:chi-covers-bdry}]
Observe that $X\subseteq |A'|$  to $X\subseteq |A|$. Indeed, since $|A|$ is closed 
if $X-|A|\neq \text{\O}$ then we may choose a point in $x\in X-|A|$ such that $x\not\in \partial X$,
since the projection $q$ is a homeomorphism on $X-\partial X$, we conclude that $q(x)\not\in X'-|A'|$. Next, 
 Equation \eqref{eq:chiX=chiA-rel} follows immediately from Corollary \ref{cor:chi-covers-no-bdry}, the fact that $A$ and $A_{\partial X}$ are good and the identities 
\[
 \chi(|A'|)=\chi_{rel}(|A|,|A|\cap \partial X),\quad \chi(X')=\chi_{rel}(X,\partial X).
\]
Now, thanks to \eqref{eq:chi-rel} we compute
\[
\begin{split}
 \chi_{rel}(X,\partial X) & = \chi(X)-\#\{\partial X\},\\
 \chi_{rel}(|A|,|A|\cap \partial X) & = \chi(|A|)-\#\{|A|\cap \partial X\},
\end{split}
\]
which yields \eqref{eq:chiX=chiA-rel-partial}. 
\end{proof}

\bibliographystyle{plain}

\def\cprime{$'$}

\end{document}